\documentclass[draft]{birkmult}

\usepackage{amsmath, amsthm, amssymb, mathrsfs}
\usepackage{epsfig, graphicx}

\newtheorem{thm}{Theorem}
\newtheorem{prop}{Proposition}

\newtheorem*{thmA}{Theorem A}
\newtheorem*{thmB}{Theorem B}
\newtheorem*{thmC}{Theorem C}
\newtheorem*{thmI}{The Index Theorem}
\newtheorem*{df}{Definition}

\numberwithin{equation}{section}

\newcommand{\T}		{\mathbb{T}}
\newcommand{\D}		{\mathbb{D}}
\newcommand{\R}		{\mathbb{R}}
\newcommand{\C}		{\mathbb{C}}
\newcommand{\N}		{\mathbb{N}}

\newcommand{\A}		{\mathbb{A}}
\newcommand{\ed}	{\omega_{[a,b]}}
\newcommand{\pj}		{\textnormal{P}}
\newcommand{\gm}		{G}
\newcommand{\szf}		{S}
\newcommand{\poly}	{\mathcal{P}}
\newcommand{\mpoly}	{\mathcal{M}}
\newcommand{\rat}		{\mathcal{R}}
\newcommand{\meas}	{\mathbf{M}}
\newcommand{\map}	{\Psi}
\newcommand{\eba}	{\Phi}
\newcommand{\cmap}	{\phi}
\newcommand{\imap}	{\varphi}
\newcommand{\f}		{\mathfrak{f}}
\newcommand{\E}		{\mathscr{E}}
\newcommand{\wn}		{\mathbf{w}_\T}
\newcommand{\da}		{D_\f}
\newcommand{\supp}	{\textnormal{supp}}

\newcommand{\im}		{\textnormal{Im}}
\newcommand{\re}		{\textnormal{Re}}
\newcommand{\const}	{\textnormal{const.}}

\begin{document}

\title[Asymptotic Uniqueness of Best Rational Approximants]{\LARGE Asymptotic Uniqueness of Best Rational Approximants to Complex Cauchy Transforms in ${L}^2$ of the Circle}

\author[L. Baratchart]{Laurent Baratchart}

\address{INRIA, Project APICS \\
2004 route des Lucioles --- BP 93 \\
06902 Sophia-Antipolis, France}

\email{laurent.baratchart@sophia.inria.fr}

\author[M. Yattselev]{Maxim Yattselev}

\address{INRIA, Project APICS \\
2004 route des Lucioles --- BP 93 \\
06902 Sophia-Antipolis, France}

\email{myattsel@sophia.inria.fr}

\date{\normalsize \today}

\begin{abstract}
For all $n$ large enough, we show uniqueness of a critical point in best rational  approximation of degree $n$, in the $L^2$-sense on the unit circle, to functions of the form
\[
\f(z) = \int\frac{d\mu(t)}{z-t}~+~r(z), \quad d\mu = \dot\mu \,d\ed,
\]
with $r$ a rational function and $\dot\mu$ a complex-valued Dini-continuous function on a real segment $[a,b]\subset(-1,1)$ which does not vanish, and whose argument is of bounded variation. Here $\ed$ stands the normalized arcsine distribution on $[a,b]$.
\end{abstract}



\maketitle

\noindent
{\it Dedicated to Guillermo Lop\`ez Lagomasino on the occasion of his 60-th birthday}

\section{Introduction}
Best rational approximation of given degree to a holomorphic function,
in the least squares sense 
on the boundary of a disk included in the domain of analyticity, is a 
classical issue for which early references are 
\cite{Erohin,Walsh,Levin,DJ,DD}. The interplay between complex and 
Fourier analysis induced by the circular symmetry confers to
such an approximation a natural character, 
and the corresponding approximants provide one with nice examples of  
locally convergent sequences of diagonal
multipoint Pad\'e interpolants. The problem can be recast as best
rational approximation of given degree in the Hardy space $H^2$ of the 
unit disk and also, upon reflecting the functions involved 
across the unit circle, in the Hardy space of the 
complement of the disk which is the framework we shall really work with.

Because of the natural isometry between Hardy spaces of the disk and
the half-plane, that preserves rationality and the degree
\cite[Ch. 8]{Hoffman}, the question can equivalently be stated as best 
rational $L^2$-approximation
of given degree on the line to a function holomorphic in a half-plane.

Further motivation for this type of approximation stems from Control Theory and
Signal Processing. Indeed, the transfer-function of
a stable linear control system belongs to the Hardy space of
the half-plane or of the complement of the disk, depending whether the 
setting is in continuous or discrete time, and it is rational if the system is 
finite-dimensional. Moreover, by Parseval's theorem,
the $L^2$ norm on the line or the circle of this transfer function 
coincides with the norm of the underlying convolution operator from
$L^2[0,\infty)$ to $L^\infty[0,\infty)$ in the time domain
\cite{DFT}. Further, in a stochastic context, it coincides 
with the variance of the output
when the input is white noise \cite{HD}.
Therefore approximating the transfer function by a 
rational function of degree $n$, in $L^2$ of the line or the circle,
is tantamount to identify the best system of order $n$ to model 
the initial system with respect to the criteria just mentioned.
Also, since any stationary regular stochastic process is the output of a 
linear control system 
fed with white noise \cite{Shiryayev,Doob}, this approximation
yields the best ARMA-process to model the initial process 
while minimizing the variance of the error. A thorough discussion of 
such connections with System Theory, as well as additional references, 
can be found in \cite{B_CMFT99}.

From the constructive viewpoint no algorithm is known to constructively
solve the 
question we raised, and from a computational perspective this is a typical 
non-convex minimization problem whose numerical solution is often hindered 
by the occurence of local minima. It is therefore of major
interest in practice to establish conditions on the function to be 
approximated  that ensure uniqueness of a local minimum. This turns out 
to be difficult, like most uniqueness issues in nonlinear 
approximation.

New ground for the subject was broken in \cite{thB}, where a 
differential-topological method was introduced to approach the
uniqueness issue for \emph{critical points}, {\it i.e.}
stationary points of the 
approximation error. Uniqueness of a critical point implies 
uniqueness of a local minimum, but is a stronger property 
which is better suited to analysis. In fact, the above-mentioned method 
rests on the so-called  {\it index theorem}
\cite{BO88} that provides us with a relation between the Morse indices of the 
critical points, thereby reducing the proof of uniqueness, 
which is a  global property, to checking that each critical point has Morse
index 0, which is a local issue. The latter is in turn
equivalent to each critical point being a non-degenerate local minimum.

This approach was taken in \cite{BW93}
to handle the case where the approximated function is of
Markov type, that is,
the Cauchy transform of a positive measure on a real segment, 
when that measure is supported within some absolute bounds.
Subsequently, in \cite{BSW96}, the property of being a 
local minimum was connected to classical interpolation theory and the 
technique was 
applied to prove asymptotic uniqueness of a critical point in best 
$L^2$ rational approximation to $e^{1/z}/z$ on the unit circle, as well as 
in best $L^2$ rational approximation to generic holomorphic functions 
over small circles; 
here, asymptotic uniqueness means 
uniqueness in degree $n$ for all sufficiently large $n$.
The criterion derived in \cite{BSW96} for being a local minimum was
further refined in 
\cite{BStW01b}, where it is shown that asymptotic uniqueness of a 
critical point  holds for Markov
functions whose defining measure satisfies the Szeg\H{o} condition. The result 
is sharp in that the Szeg\H{o} condition cannot be omitted in general
\cite{BStW99}. A general condition on the logarithmic derivative of 
the approximated function was derived \cite{B06} but
it only ensures uniqueness in degree 1. A criterion for best approximation in 
degree $n-1$ to a rational function of degree $n$ can further be found in 
\cite[Thm. 9.1]{BS02}, based on fast geometric decay of the error in 
lower degree. 

Altogether, these works indicate the fact, perhaps unexpected, that uniqueness
of a critical point in best $L^2$ rational approximation is linked to a
regular decrease of the error.

The present paper can be viewed as a sequel to \cite{BStW01b}. Indeed,
the latter 
reference expressed hope that the techniques set up there
could be adapted to handle more general Cauchy integrals than Markov 
functions. Below, we take a step towards carrying out this program.
Specifically, we consider Cauchy transforms of 
\emph{complex}  measures that are absolutely continuous
with respect to the equilibrium distribution on a real segment 
$[a,b]\subset(-1,1)$. The density will be required to be Dini-smooth
and non-vanishing. In addition, it
should admit an argument function of bounded variation on $[a,b]$.
Moreover, we handle with little extra-pain
the case where a rational function
is added to such a Cauchy tranform.

For functions of this kind, we establish an analog to 
\cite[Thm. 1.3]{BStW01b}, namely
asymptotic uniqueness of a critical point in best rational 
approximation for the  $L^2$-norm on the unit circle.
This is the first uniqueness result in degree greater than 1 
for Cauchy integrals with complex densities, more generally
for non-rational functions without conjugate symmetry. 
In contrast, say,  to \cite[Thm. 3]{BW93}, 
it is only fair to say that 
such a statement is not really constructive in that 
no estimate is provided for the degree beyond which uniqueness
prevails. However, considering our restricted knowledge on
uniqueness in non-linear complex approximation, our result
sheds considerable light on the behaviour of best rational 
approximants to Cauchy transforms, and it is our hope that 
suitable refinements of the technique will eventually produce
effective bounds.

Our method of proof follows the same pattern as \cite{BStW01b}. Namely,
the index theorem is invoked to reduce the question of uniqueness to whether 
each critical point is a non-degenerate local minimum. Next, a criterion for 
being a local minimum is set up, based on 
a comparison between the error function generated
by the critical point under examination and the error function attached to
a particular multipoint Pad\'e interpolant of lower degree; we call it for
short the \emph{comparison criterion}.
Finally, the fact this
criterion applies when the degree is sufficiently large depends on strong 
asymptotic formulas for the error in rational interpolation to functions of
the type we consider, that were recently obtained in 
\cite{BY09b,uBY3,uY2}, and on a specific design
of interpolation nodes to 
build the particular multipoint Pad\'e interpolant of lower degree
that we need.

With respect to \cite{BStW01b}, however, two main differences arise.
The first is that the comparison criterion was set up there
for conjugate-symmetric functions only, {\it i.e.} for those functions
having real Fourier  coefficients. Because we now  
adress Cauchy transforms of complex densities on a segment,
we handle complex Fourier coefficients as well. Although the 
corresponding changes are mostly mechanical, they have to be
carried out thoroughly for they impinge on the computation of the Hessian
quadratic form and on the nondegeneracy thereof.

The second difference causes more serious difficulties. Indeed,
the construction of the special interpolant of lower degree 
needed to apply the comparison criterion
requires rather precise control on the poles of multipoint Pad\'e interpolants
to the approximated function. For Markov functions, it is known that
such poles are the zeros of  certain orthogonal polynomials with respect to 
a positive measure on the segment $[a,b]$, therefore they lie on 
that segment. But for Cauchy transforms of complex measures, 
the poles are the zeros of some
\emph{non-Hermitian} orthogonal polynomial with respect to a complex 
measure on $[a,b]$, and their behaviour does not lend itself to analysis 
so easily. We resort here to the work in \cite{BKT05} on the geometry
of non-Hermitian orthogonal polynomials
to overcome this difficulty, and this is where the boundedness of the 
variation of the density's argument becomes important.

Let us briefly indicate 
some generalizations of our results that were not included
here.
Firstly, asymptotic uniqueness of a critical point
extends to Cauchy transforms of 
absolutely continuous measures whose density with respect to 
\linebreak {\it Lebesgue} (rather than equilibrium) measure
satisfies similar assumptions, {\it e.g.} non-vanishing and 
H\"older-smoothness; in fact, densities with respect to
\emph{any} Jacobi weight could be handled the same way under suitable 
regularity requirements. We did not mention them, however, because
such an extension depends on 
asymptotics for non-Hermitian orthogonal polynomials with respect to 
this type of weight which are yet unpublished \cite{uBY5}. 

Secondly, finitely many zeros in the density would still be acceptable,
provided they are of power type with sufficiently small exponent. 
Developing the precise estimates would make the paper heavier
(compare \cite[Thm. 4]{uBY3} and \cite[Thm. 5]{uY2}), so we felt
better omitting this stronger version.

The organization of the article is as follows. In Section
\ref{sec:intro} we present the rational approximation problem under study
and we state the main result of the paper, which is Theorem \ref{thm:main}.
Section \ref{sec:points} introduces the critical points
in $H^2$-rational approximation 
and develops their interpolating properties;
this part is adapted to complex Fourier coefficients
from \cite{BSW96}.
The index theorem and the comparison criterion are expounded
in section \ref{sec:criterion}, paralleling the treatment for real Fourier 
coefficients given in \cite{BStW01b}.
Section \ref{sec:interpol}
recalls the necessary material on interpolation from \cite {BY09b,uBY3,uY2},
which is needed to carry out the comparison between critical points and 
interpolants of lower degree required in the comparison 
criterion.
Finally, elaborating on \cite{BStW01b}, we prove Theorem \ref{thm:main} in 
Section \ref{sec:proof}.

\section{Notation and Main Results}
\label{sec:intro}

Let $\T$ be the unit circle and $\D$ the unit disk. We let $L^2$ (resp. $L^\infty$) stand for the space of square-integrable (resp. essentially bounded) measurable functions on $\T$. Denote by $H^2$ the familiar Hardy space of the unit disk, consisting of those $L^2$-functions whose  Fourier coefficients of 
strictly negative index are zero. The space $H^2$ identifies with traces on $\T$ of those holomorphic functions in $\D$ whose $L^2$-means on circles centered at zero with radii less than 1 are uniformly bounded  above. In fact, any such function has non-tangential boundary values almost everywhere on $\T$ that define a member of $L^2$ from which the function can be recovered by means of a Cauchy integral \cite{Duren}. The space of bounded holomorphic functions on $\D$ is denoted by $H^\infty$ and is endowed with the $L^\infty$ norm of the trace.

Let $\bar H_0^2$ be the orthogonal complement of $H^2$ in $L^2$ with respect to the standard scalar product
\begin{equation}
\label{defscal}
\langle f,g \rangle := \int_\T f(\tau)\overline{g(\tau)}\,\frac{|d\tau|}{2\pi}, \quad f,g\in L^2.
\end{equation}
The space $\bar H_0^2$, in turn, identifies with traces of those holomorphic functions in $\overline{\C}\setminus\overline{\D}$ that vanish at  infinity and whose $L^2$-means on circles centered at zero with radii greater then 1 are uniformly bounded above. In what follows, we denote by $\|\cdot\|_2$ the norm on $L^2$, $H^2$, and $\bar H_0^2$ induced by the scalar product \eqref{defscal}. On one occasion, we shall refer to the Hardy space $\bar H_0^2(\overline{\C}\setminus\overline{\D_\rho})$, which is defined similarly except that $\D$ gets replaced by $\D_\rho:=\{|z|<\rho\}$ where $\rho>0$.

Set $\poly_n$ for the space of algebraic polynomials of degree at most $n$ and $\mpoly_n$ for the subset of monic polynomials with exactly $n$ zeros in $\D$. Note that $q\in\mpoly_n$ if and only if $1/q\in\bar H_0^2$ and $1/q(z)\sim z^{-n}$ at infinity. From the differential viewpoint, we regard $\poly_n$ as $\C^{n+1}$ and $\mpoly_n$ as an open subset of $\C^n$, upon taking the coefficients as coordinates (except for the leading coefficient in $\mpoly_n$ which is fixed to unity).

Define 
\begin{equation}
\label{defRn}
\rat_n := \left\{\frac{p(z)}{q(z)}=\frac{p_{n-1}z^{n-1}+p_{n-2}z^{n-2}+\cdots+p_0}{z^n+q_{n-1}z^{n-1}+\cdots+q_0}:~ p\in\poly_{n-1},~q\in\mpoly_n\right\}.
\end{equation}
It is easy to check that $\rat_n$ consists of those rational functions of degree at most $n$ that belong to $\bar H_0^2$, and we endow it with the corresponding topology. Coordinatizing $\poly_{n-1}$ and $\mpoly_n$ as above, it is straightforward to see that the canonical surjection $J:\poly_{n-1}\times\mpoly_n\to\rat_n$ is smooth ({\it i.e.}, infinitely differentiable) when viewed as a $\bar H_0^2$-valued map. Note that $J$ is not injective, due to possible cancellation between $p$ and $q$, but it is a local homeomorphism at every pair $(p,q)$ such that $p$, $q$ are coprime.

We shall be concerned with the following problem.

\smallskip
\noindent
{\bf Problem 1:}
\noindent
{\it Given $f\in\bar H_0^2$ and $n\in\N$,  find $r\in\rat_n$ to minimize $\|f-r\|_2$.}
\smallskip

Let us point out two equivalent formulations of Problem 1 that account for early discussion made in the introduction.

Firstly, it is redundant to assume that $r$ lies in $\bar H_0^2$, as is subsumed in the definition of $\rat_n$. 
Indeed, by partial fraction expansion, a rational function of degree at most $n$ in $L^2$ can be written  as $r_1+r_2$, where $r_1\in\bar H^2_0$ and $r_2\in H^2$ have degree at most $n$. By orthogonality of $\bar H^2_0$ and $H^2$, we get $\|f-r\|_2^2=\|f-r_1\|_2^2+\|r_2\|_2^2$ so that $r_1$ is a better candidate approximant than $r$, showing that Problem 1 is in fact equivalent to best rational approximation of given degree to $f$ in $L^2$.

Secondly, composing with $z\mapsto1/z$, which is an $L^2$-isometry mapping $\bar H_0^2$ onto $zH^2$ while preserving rationality and the degree, Problem 1 transforms to best approximation in $H^2$ of functions vanishing at the origin by rational functions of degree at most $n$ that vanish at the origin as well. However, by Parseval's identity, any best rational approximant to $g$ in $H^2$ has value $g(0)$ at $0$. Therefore Problem 1 is equivalent to 

\smallskip
\noindent
{\bf Problem 2:}
\noindent
{\it Given $g\in H^2$ and $n\in\N$,  find a rational $r$  of degree at most $n$ in $H^2$ to minimize $\|f-r\|_2$.}
\smallskip

Problem 1 is the one we shall work with, and we refer to it as the \emph{ rational $\bar H_0^2$-approximation problem to $f$ in degree} $n$. It is well-known (see \cite[Prop. 3.1]{B06} for a proof and further bibliography on the subject) that the minimum is attained and that a minimizing $r$, called a {\it best rational approximant of degree $n$ to} $f$, lies in $\rat_n\setminus\rat_{n-1}$ unless $f\in\rat_{n-1}$. Uniqueness of such an approximant is a delicate matter. Generically, there is only one best approximant by a theorem of Stechkin on Banach space approximation from approximately compact sets \cite{Braess}. However, the proof is non-constructive and does not allow us to determine which functions have a unique best approximant and which functions do not. Moreover, from the computational viewpoint, the main interest lies not so much with uniqueness of a best approximant, but rather with uniqueness of a local best approximant for such places are all what a numerical search can usually spot. By definition, a {\it local best approximant} is a function $r_l\in\rat_n$ such that $\|f-r_l\|_2\leq\|f-r\|_2$ for all $r\in\rat_n$ in some neighborhood of $r_l$. Like best approximants, local best approximants lie in $\rat_n\setminus\rat_{n-1}$ unless $f\in\rat_{n-1}$ \cite{B06}.

Still more general is the notion of a \emph{critical point}, which is defined as follows. Fix $f\in\bar H_0^2$ and put
\begin{equation}
\label{eq:ef}
\begin{array}{rll}
\eba_{f,n}:\rat_n &\to& [0,\infty) \\
r &\mapsto& \|f-r\|_2^2.
\end{array}
\end{equation}
A pair $(p,q)\in\poly_{n-1}\times\mpoly_n$ is called {\it critical} if all partial derivatives of $\eba_{f,n}\circ J$ vanish at $(p,q)$. Subsequently,  a rational function $r_c\in\rat_n$ is said to be a {\it critical point} of $\eba_{f,n}$ if there is a critical pair $(p_c,q_c)$ such that $r_c=p_c/q_c$. Critical points fall into two classes: they are termed \emph{irreducible} if they have exact degree $n$, and \emph{reducible} if they have degree strictly less than $n$. Note that $r_c$ is irreducible if and only if $p_c$ and $q_c$ are coprime in some (hence any) representation $r_c=p_c/q_c$. Clearly a local best approximant is a particular instance of a critical point, and it is irreducible unless $f\in\rat_{n-1}$.

In the present work, we dwell on a differential topological approach to uniqueness of a critical point introduced in \cite{thB} and further developed in \cite{BO88,B_CMFT99}. In this approach, global uniqueness is deduced from local analysis of the map $\eba_{f,n}$. Specifically, to conclude there is only one critical point, which is therefore the unique local minimum (and {\it a fortiori} the global minimum as well), one needs to show that each critical point is irreducible, does not interpolate the approximated function on $\T$, and is a nondegenerate local minimum; here \emph{nondegenerate} means that the second derivative is a nonsingular quadratic form. This method turns out to be fruitful when studying rational approximation to Cauchy transforms of measures supported in $(-1,1)$, {\it i.e.}, functions of the form
\begin{equation}
\label{eq:ct}
f_\mu(z) := \int\frac{d\mu(t)}{z-t}, \quad \supp(\mu)\subset(-1,1).
\end{equation}
The first result in this direction was obtained in \cite[Thm. 3]{BW93} when $f_\mu$ is a \emph{Markov function}, meaning that $\mu$ in (\ref{eq:ct}) is a positive measure.  It goes as follows.

\begin{thmA}
Let $\mu$ be a positive measure supported on $[a,b]\subset(-1,1)$ where $a$ and $b$ satisfy $b-a \leq \sqrt2 \left( 1 - \max\left\{a^2,b^2\right\} \right)$. Assume further that $\mu$ has at least $n$ points of increase, {\it i.e.}, $f_\mu\notin\rat_{n-1}$. Then there is a unique critical point in rational $\bar H_0^2$-approximation of degree $n$ to $f_\mu$.
\end{thmA}

Removing the restriction on the size of the support makes the situation more difficult. The following theorem \cite[Thm. 1.3]{BStW01b} asserts that rational approximants are asymptotically unique for Markov functions whose defining measure is sufficiently smooth. Hereafter, we denote by $\ed$ the normalized arcsine distribution on $[a,b]$ given by $d\ed(t)=dt/(\pi\sqrt{(t-a)(b-t)})$.

\begin{thmB}
Let $\mu$ be a positive measure supported on $[a,b]\subset(-1,1)$ and let us write $d\mu = \mu^\prime dt + d\mu_s$, where $\mu_s$ is singular and $\mu^\prime$ is integrable on $[a,b]$. If $\mu$ satisfies the Szeg\H{o} condition: $\int \log\mu^\prime d\ed > - \infty$, then there is a unique critical point in rational $\bar H_0^2$-approximation of degree $n$ to $f_\mu$ for all $n$ large enough.
\end{thmB}

As an additional piece of information, the following negative result \cite[Thm. 5]{BStW99} shows that the asymptotic nature of the previous theorem is indispensable.

\begin{thmC}
For each $n_0\in\N$ there exists a positive measure $\mu$ satisfying the Szeg\H{o} condition such that for each odd $n$ between 1 and $n_0$ there exist at least two different best rational approximants of degree $n$ to $f_\mu$.
\end{thmC}

% The criterion of asymptotic uniqueness derived in \cite[Thm. 5.5]{BStW01b} (see also Section \ref{sec:criterion}), which was used to prove Theorem B, is stated for any function in $\bar H_0$. Hence it is independent of the nature of Markov functions or more generally Cauchy transforms. In fact, a slightly weaker criterion was obtained in \cite{BSW96} to deduce asymptotic uniqueness of rational approximation to $e^{1/z}/z$. In this work, we combine the uniqueness criterion from \cite{BStW01b} with the results in \cite{uY2} (see also Section \nolinebreak \ref{sec:interpol}) on the asymptotic behavior of rational approximants to Cauchy transforms to obtain their asymptotic uniqueness. 
Our goal is to extend Theorem B to a class of \emph{complex measures} which is made precise in the definition below. Recall that a function $h$  with modulus of continuity $\omega_h$ is said to be \emph{Dini-continuous} if $\omega_h(t)/t$ is integrable on $[0,\epsilon]$ for some (hence any) $\epsilon>0$.

\begin{df}[Class $\meas$] A measure $\mu$ is said to belong to the class $\meas$ if $\supp(\mu)\subset(-1,1)$ is an interval, say $[a,b]$, and $d\mu=\dot\mu d\ed$, where $\dot\mu$ is a Dini-continuous non-vanishing function with an argument of bounded variation on $[a,b]$.
\end{df}

Observe that we deal here with $\dot\mu$, the Radon-Nikodym derivative of $\mu$  with respect to the arcsine distribution, rather then with $\mu^\prime$, the Radon-Nikodym derivative with respect to the Lebesgue measure. Our main result is:

\begin{thm}
\label{thm:main}
Let $\f:=f_\mu+r$, where $\mu\in\meas$ and $r\in\rat_m$ has no poles on $\supp(\mu)$. Then there is a unique critical point in rational $\bar H_0^2$-approximation of degree $n$ to $\f$ for all $n$ large enough.
\end{thm}

Before we can prove the theorem, we must study in greater detail the structure of critical points, which is the object of Sections \ref{sec:points} and \ref{sec:criterion} to come.

\section{Critical Points}
\label{sec:points}

The following theory was developed in \cite{BO88,BOW90,BOW92,BSW96} when the function $f$ to be approximated is conjugate-symmetric, {\it i.e.}, $f(\bar z)=\overline{f(z)}$, and the rational approximants are seeked to be conjugate-symmetric as well. In other words, when a function with real Fourier-Taylor expansion at infinity gets approximated by a rational function with real coefficients. Surprisingly enough, this is \emph{not} subsumed in Problem 1 in that conjugate-symmetric functions need not have a best approximant out of $\rat_n$ which is conjugate-symmetric. For Markov functions, though, it is indeed the case~\cite{BKT05}. Below, we develop an analogous theory for Problem 1, that is, without conjugate-symmetry assumptions. This involves only  technical modifications of a rather mechanical nature.

Hereafter, for any $f\in L^2$, we set $f^\sigma(z) := (1/z)\overline{f(1/\bar z)}$. Clearly, $f\to f^\sigma$ is an isometric involution mapping $H^2$ onto $\bar H_0^2$ and vice-versa. Further, for any $p\in\poly_k$, we set $\check p(z):=\overline{p(1/\bar z)}$ and define its reciprocal polynomial (in $\poly_k$) to be $\widetilde p(z):=z^k\check p(z)=z^k\overline{p(1/\bar z)}$. Note that $\widetilde{p}$ has the same modulus as $p$ on $\T$ and its zeros are reflected from those of $p$ across $\T$.

\subsection{Critical Points as Orthogonal Projections}

Fix $f\in H_0^2$ and let $\eba_n:=\eba_{f,n}$ be given by (\ref{eq:ef}). It will be convenient to use complex partial derivatives with respect to $p_j$, $q_k$, $\bar p_j$, $\bar q_k$, where, for $0\leq j,k\leq n-1$, the symbols $p_k$ and $q_k$ refer to the coefficients of $p\in\poly_{n-1}$ and $q\in\mpoly_n$ as in equation (\ref{defRn}). By complex derivatives we mean the standard Wirtinger operators, {\it e.g.}, if $p_j=x_j+iy_j$ is the decomposition into real and imaginary part then $\partial/\partial p_j=(\partial/\partial x_j-i\partial/\partial y_j)/2$ and $\partial/\partial \bar p_j=(\partial/\partial x_j+i\partial/\partial y_j)/2$. The standard rules for derivation are still valid, obviously $\partial g(p_j)/\partial\bar p_j=0$ if $g$ is holomorphic, and it is straightforward that $\overline{\partial h/\partial p_j}=\partial \bar h/\partial \bar p_j$ for any function $h$. In particular, since $\eba_n$ is real,
\begin{equation}
\label{dconj}
\frac{\partial\eba_n}{\partial \bar p_j} = \overline{\left(\frac{\partial\eba_n}{\partial p_j}\right)} \quad \mbox{and} \quad \frac{\partial\eba_n}{\partial \bar q_k} = \overline{\left(\frac{\partial\eba_n}{\partial q_k}\right)}.
\end{equation}
Thus, writing $\eba_n\circ J(p,q)=\langle f-p/q,f-p/q\rangle$ and differentiating under the integral sign, we obtain that a critical pair $(p_c,q_c)$ of $\eba_n\circ J$ is characterized by the relations
\begin{eqnarray}
\label{eq:partials1}
\frac{\partial\eba_n}{\partial p_j}\left(p_c,q_c\right) &=& \left\langle \frac{z^j}{q_c},f-\frac{p_c}{q_c}\right\rangle =  0, \quad j\in\{0,\ldots,n-1\}, \\
\label{eq:partials2}
\frac{\partial\eba_n}{\partial q_k}\left(p_c,q_c\right) &=& - \left\langle \frac{z^kp_c}{q_c^2},f-\frac{p_c}{q_c}\right\rangle = 0, \quad k\in\{0,\ldots,n-1\}.
\end{eqnarray}
Equation (\ref{eq:partials1}) means that $p_c/q_c$ is the orthogonal projection of $f$ onto $V_{q_c}$, where for any $q\in\mpoly_n$ we let $V_q:=\{p/q:~p\in\poly_{n-1}\}$ to be the $n$-dimensional linear subspace of $\bar H_0^2$ consisting of rational functions with denominator $q$. In what follows, we consistently denote the orthogonal projection of $f$ onto $V_q$ by $L_q/q$, where $L_q\in\poly_{n-1}$ is uniquely characterized by the fact that
\begin{equation}
\label{eq:orth}
\left\langle f-\frac{L_q}{q},\frac{p}{q} \right\rangle = 0 \quad \mbox{for any} \quad p\in\poly_{n-1}.
\end{equation}
Taking equation (\ref{eq:partials2}) into account, we conclude from what precedes that critical points of $\eba_n$ are precisely  $\rat_n$-functions of the form $L_{q_c}/q_c$, where $q_c\in\mpoly_n$ satisfies
\begin{equation}
\label{critphiq}
\left\langle\frac{z^kL_{q_c}}{{q_c}^2},f-\frac{L_{q_c}}{q_c}\right\rangle=0, 
\quad k\in\{0,\ldots,n-1\}.
\end{equation}
Now, it is appearent from (\ref{eq:orth}) that $L_q$ is a smooth function of $q$, therefore we define a smooth map $\map_n$ on $\mpoly_n$
by setting
\begin{equation}
\label{eq:psif}
\begin{array}{rll}
\map_n=\map_{f,n}:\mpoly_n &\to&  [0,\infty) \\
q &\mapsto& \| f-L_q/q \|_2^2
\end{array}.
\end{equation}
By construction, $\eba_n$ attains a local minimum at $r=L_{q_l}/q_l$ if and only if $\map_n$ attains a local minimum at $q_l$, and the assumed values are the same. More generally, $r\in\rat_n$ is a critical point of $\eba_n$ if and only if $r=L_{q_c}/q_c$, where $q_c\in\mpoly_n$, is a critical point of $\map_n$. This is readily checked upon comparing (\ref{critphiq}) with the result of the following computation:
\begin{eqnarray}
\frac{\partial\map_n}{\partial q_k}(q) &=& \left\langle\frac{(\partial L_q/\partial q_k)q-z^kL_q}{q^2},f-\frac{L_q}{q}\right\rangle + \left\langle f-\frac{L_q}{q},-\frac{(\partial L_q/\partial \bar q_k)}{q}\right\rangle \nonumber \\
\label{eq:partials}
{} &=& -\left\langle\frac{z^kL_q}{q^2},f-\frac{L_q}{q}\right\rangle, \quad k\in\{0,\ldots,n-1\},
\end{eqnarray}
where we applied (\ref{eq:orth}) using that the derivatives of $L_q$ lie in $\poly_{n-1}$ and that $\partial q/\partial \bar q_k=0$. For simplicity, we drop from now on the subscript ``$c$'' we used so far as a mnemonic for ``critical''. Altogether we proved the following result: 
\begin{prop}
\label{prop:1}
For $f\in\bar H_0^2$, let $\eba_n$ and $\map_n$ be defined by (\ref{eq:ef}) and (\ref{eq:psif}), respectively. Then $r\in\rat_n$ is a critical point of $\eba_n$ if and only if $r=L_{q}/q$ and $q\in\mpoly_n$ is a critical point of $\map_n$.
\end{prop}
In view of Proposition \ref{prop:1}, we shall extend to $\map_n$ the terminology introduced for $\eba_n$ and say that a critical point $q\in\mpoly_n$ of $\map_n$ is irreducible if $L_q$ and $q$ are coprime. 

\subsection{Interpolation Properties of Critical Points}
\label{subseccrit}

If we denote with a superscript ``$\perp$'' the orthogonal complement in $\bar H_0^2$,  it is elementary to check that
\begin{equation}
\label{eq:vqperp}
V_q^\perp =\left\{\frac{\widetilde q}{q}u:~u\in\bar H_0^2\right\}, \quad \bar H_0^2 = V_q\oplus V_q^\perp.
\end{equation}
Hence, by (\ref{eq:orth}), there exists $u_q = u_{f,q}\in \bar H_0^2$ such that
\begin{equation}
\label{eq:ip}
fq - L_q = \widetilde q u_q.
\end{equation}
Relation (\ref{eq:ip}) means that $L_q/q$ interpolates $f$ at the reflections of the zeros of $q$ across $\T$. Assume now that $q\in\mpoly_n$ is a critical point of $\map_n$. Then, combining (\ref{eq:ip}) and (\ref{eq:partials}), we derive that 
\begin{equation}
\label{tradint}
0 = \left\langle \frac{pL_q}{q^2}, f-\frac{L_q}{q} \right\rangle = \left\langle \frac{pL_q}{q^2}, \frac{\widetilde q u_q}{q} \right\rangle = \int_\T\frac{p(\tau)}{q(\tau)} \frac{(L_qu_q^\sigma)(\tau)}{\widetilde q(\tau)}\frac{d\tau}{2\pi i},\qquad p\in\poly_{n-1}.
\end{equation}
Since $L_qu_q^\sigma/\widetilde q\in H^2$, we see by letting $p$ range over elementary divisors of $q$ and applying the residue formula that (\ref{tradint}) holds if and only if each zero of $q$ is a zero of $L_qu_q^\sigma$ of the same multiplicity or higher. That is, $q$ is a critical point of $\map_n$ if and only if $q$ divides $L_qu_q^\sigma$ in $H^2$.
%In particular, $q$ divides $u_q^\sigma$ in $H^2$ if it is an irreducible
%critical point.
% $q(z)/z^n$ divides $u_q(z)$ in $\bar H_0^2$. In view of (\ref{eq:ip}),
% this means that $L_q/q$ interpolates $f$ \emph{with order 2} at the 
% reflections of the zeros of $q$ across $\T$, a property which is
% well-known to hold for best approximants \cite{Levin}. 

Let $d\in\mpoly_k$ be the monic g.c.d. of $L_q$ and $q$, with $0\leq k\leq n-1$. Writing $L_q = dp^*$ and $q=dq^*$, where $p^*$ and $q^*$ are coprime, we deduce that $q^*$ divides $u_q^\sigma$ in $H^2$ or equivalently that $u_q=\check{q^*} h$ for some $h\in\bar H_0^2$. Besides, it follows from (\ref{eq:orth}), applied with $p=dv$ and $v\in\poly_{n-k-1}$, that $p^*=L_{q^*}$. Therefore, upon dividing (\ref{eq:ip}) by $d$, we get 
\begin{equation}
\label{interpw}
fq^*-L_{q^*}=\frac{\widetilde{q}\check{q^*}}{d} h,
\end{equation}
implying that $u_{q^*}=\widetilde{d}\check{q^*} h/d$. In particular, $q$, thus {\it a fortiori} $q^*$, divides $u_{q^*}^\sigma$ in $H^2$, whence $q^*$ is critical for $\map_{n-k}$ by what we said before. Finally, dividing (\ref{interpw}) by $q^*$ and taking into account the definition of the reciprocal polynomial, we find that  we established the following result.
% Moreover, combining linearly equations (\ref{eq:partials}), we obtain 
% \begin{equation}
% \label{eq:rip}
% 0 = \left\langle\frac{z^jdL_q}{q^2},f-\frac{L_q}{q}\right\rangle = \left\langle\frac{z^j L_{q^*}}{(q^*)^2},f-\frac{L_{q^*}}{q^*}\right\rangle, \quad j\in\{0,\ldots,n-k-1\},
% \end{equation}
% showing that $q^*$ is an irreducible critical point of $\map_{n-k}$.
\begin{prop}
\label{prop:2}
Let $q$ be a critical point of $\map_n$ and $d\in\mpoly_k$ be the monic g.c.d. of $L_q$ and $q$ with $0\leq k\leq n-1$. Then $q^*=q/d\in\mpoly_{n-k}$ is an irreducible critical point of $\map_{n-k}$, and $L_{q^*}/q^*$ interpolates $f$ at the zeros of $\check{q^*}^2\check d/z$ in Hermite's sense on $\overline{\C}\setminus\overline{\D}$, that is, counting multiplicities including at infinity.
\end{prop}

The converse is equally easy: if $q^*$ is an irreducible critical point of $\map_{n-k}$, and $L_{q^*}/q^*$ interpolates $f$ at the zeros of $\check{q^*}^2\check d/z$ in Hermite's sense for some $d\in\mpoly_k$, then $q=q^*d$ is a critical point of $\varphi_n$ and $d$ is the monic g.c.d. of $q$ and $L_q$. This we shall not need.

It is immediately seen from Proposition \ref{prop:2} that a critical point of $\eba_{f,n}$ must interpolate $f$ with order 2 at the reflections  of its poles across $\T$; for best approximants, this property is classical \cite{Erohin,Levin}. 

\subsection{Smooth Extension of $\map_n$}
\label{subsec:3}

One of the advantages of $\map_n$, as compared to $\eba_n$, is that its domain of definition can be compactified, which is essential to rely on methods from differential topology. To do that, however, we need to place an additional requirement on $f$.

Let us denote by $\bar H_0$ the subset of $\bar H_0^2$ comprised of functions that extend holomorphically across $\T$. \emph{Hereafter we will suppose that} $f\in\bar H_0$ and pick $\rho=\rho(f)<1$ such that $f$ is holomorphic in $\{|z|>\rho-\epsilon\}$ for some $\epsilon>0$. In particular, $f$ is holomorphic across $\T_\rho:=\{|z|=\rho\}$.

Denote by $\overline\mpoly_n$ and $\mpoly_n^{1/\rho}$ respectively the closure of $\mpoly_n$ and the set of monic polynomials with zeros in $\D_{1/\rho}:=\{|z|<1/\rho\}$; as usual,  we regard these as subsets of $\C^n$ when coordinatized by their coefficients except the leading one. This way $\mpoly_n^{1/\rho}$ becomes an open neighborhood of the compact set $\overline\mpoly_n$, which is easily seen to consist of polynomials with zeros in $\overline\D$. Also, $q$ lies on the boundary $\partial\overline\mpoly_n=\overline\mpoly_n\setminus\mpoly_n$ if and only if it is a monic polynomial of degree $n$  having at least one zero of modulus 1 and no zero of modulus strictly greater then 1.

For $q\in\mpoly_n$, since $q/\widetilde q$ is unimodular on $\T$, it follows from (\ref{eq:ip}) that
\begin{equation}
\label{eq:anlook}
\map_n(q) = \left\|f-\frac{L_q}{q}\right\|_2^2 = \|u_q\|_2^2 = \int_\T (u_qu_q^\sigma)(\tau)\frac{d\tau}{2\pi i}.
\end{equation}
In addition, taking into account the Cauchy formula, the analyticity of $L_q/\widetilde q$ in $\D$, the analyticity of $f$ across $\T_\rho$, and  the definition of the $\sigma$-operation, we obtain
\begin{eqnarray}
\label{eq:uq}
u_q(z)=u_{f,q}(z) &=& \frac{1}{2\pi i}\int_{\T_\rho} \frac{f(\tau)q(\tau)}{\widetilde q(\tau)} \frac{d\tau}{z-\tau}, \quad |z|>\rho, \\
\label{eq:uqs}
u_q^\sigma(z) &=& \frac{1}{2\pi i}\int_{\T_{1/\rho}} \frac{f^\sigma(\tau)\widetilde q(\tau)}{q(\tau)} \frac{d\tau}{\tau-z}, \quad |z|<1/\rho\\
\label{eq:lq}
L_q(z)=L_{f,q}(z) &=& \int_{\T_\rho} \frac{f(\tau)}{\widetilde q(\tau)}\frac{q(z)\widetilde q(\tau)-\widetilde q(z)q(\tau)}{z-\tau}\frac{d\tau}{2\pi i},  \quad |z|>\rho.
\end{eqnarray}
Now, if $q\in\mpoly_n^{1/\rho}$, then $\widetilde q$ has all its zeros of modulus greater then $\rho$,  therefore (\ref{eq:lq}) and (\ref{eq:uq}) are well-defined and smooth with respect to the coefficients of $q$, with values in $\poly_{n-1}$ and $\bar H_0^2(\overline{\C}\setminus\overline{\D_\rho})$ respectively. Because evaluation at $\tau\in\T$ is uniformly bounded with respect to $\tau$ on $\bar H_0^2(\overline{\C}\setminus\overline{\D_\rho})$, $\map_n$ in turn extends smoothly to $\mpoly_n^{1/\rho}$ in view of (\ref{eq:anlook}). Moreover, differentiating under the integral sign, we obtain
\begin{equation}
\label{eq:par}
\frac{\partial\map_n}{\partial q_j}(q) = \int_\T \left(\frac{\partial u_q}{\partial q_j}(\tau)u_q^\sigma(\tau)  + u_q(\tau)\frac{\partial u_q^\sigma}{\partial q_j}(\tau) \right) \frac{d\tau}{2\pi i},
\end{equation}
with
\begin{eqnarray}
\label{eq:par1}
\frac{\partial u_q}{\partial q_j}(z) &=& \frac{1}{2\pi i} \int_{\T_\rho} \frac{\tau^jf(\tau)}{\widetilde q(\tau)} \frac{d\tau}{z-\tau},
\quad |z|>\rho,  \\
\label{eq:par2}
\frac{\partial u_q^\sigma}{\partial q_j}(z) &=& \frac{1}{2\pi i} 
\int_{\T_{1/\rho}} \frac{-\tau^jf^\sigma(\tau)\widetilde q(\tau)}{q^2(\tau)} \frac{d\tau}{\tau-z},\quad |z|<1/\rho,
\end{eqnarray}
for $j=0,\ldots,n-1$. To recap, we have proved:
\begin{prop}
\label{prop:3}
Let $f\in\bar H_0$. Then $\map_n$ extends to a smooth function in some neighborhood of $\overline\mpoly_n$ and so do $L_q$ and $u_q$ with values in $\poly_{n-1}$ and $\bar H_0^2(\overline{\C}\setminus\overline{\D_\rho})$ respectively. In addition, (\ref{eq:par}), (\ref{eq:par1}), and (\ref{eq:par2}) hold.
\end{prop}

We shall continue to denote the extension whose existence is asserted in Proposition \ref{prop:3} by $\map_{f,n}$, or simply $\map_n$ if $f$ is understood from the context.
% In particular, if $q=vq^*$, where $v\in\poly_k$ has all its zeros on $\T$ and $q^*\in\mpoly_{n-k}$, then
% \begin{equation}
% \label{eq:rlq}
% \widetilde q = \widetilde v\widetilde q^* = e^{i\theta}v\widetilde q^* \quad \mbox{and} \quad L_q = L_{vq^*} = vL_{q^*},
% \end{equation}
% where $e^{i\theta}$ is the product of the opposites of the roots of $v$.

\subsection{Critical points on the boundary}

Having characterized the critical points of $\map_n$ on $\mpoly_n$ in Section \ref{subseccrit}, we need now to describe the critical points that it may have on $\partial\overline\mpoly_n$. We shall begin with the case where all the roots of the latter lie on $\T$.

Let $f\in\bar H_0$ and assume $v(z)=(z-\xi)^k$, $\xi\in\T$, is a critical point of $\map_k$. It immediately follows from  (\ref{eq:uq}), the Cauchy formula, and the definition of $v$ that
\begin{equation}
\label{eq:lquqT}
\widetilde v = e^{i\theta}v, \quad\mbox{where}\quad e^{i\theta} := (-\bar\xi)^k, \quad L_v \equiv 0, \quad \mbox{and} \quad u_{v} = e^{-i\theta}f.
\end{equation}
In this case, equations (\ref{eq:par1}) and (\ref{eq:par2}) become 
\begin{eqnarray}
% In particular, if $q=vq^*$, where $v\in\poly_k$ has all its zeros on $\T$ and $q^*\in\mpoly_{n-k}$, then
% \begin{equation}
% \label{eq:rlq}
\frac{\partial u_{v}}{\partial q_j}(z) &=& \frac{e^{-i\theta}}{2\pi i} \int_{\T_\rho} \frac{\tau^jf(\tau)}{v(\tau)} \frac{d\tau}{z-\tau}, \quad |z|>\rho,  \nonumber \\
\frac{\partial u_{v}^\sigma}{\partial q_j}(z) &=& \frac{-e^{i\theta}}{2\pi i} \int_{\T_{1/\rho}} \frac{\tau^jf^\sigma(\tau)}{v(\tau)} \frac{d\tau}{\tau-z}, \quad |z|<1/\rho. \nonumber
\end{eqnarray}
Plugging these expressions into (\ref{eq:par}), we obtain
\[
0 = \frac{\partial \map_{k}}{\partial q_j}(v) = \int_{\partial\A_\rho}\frac{\tau^j(ff^\sigma)(\tau)}{v(\tau)}\frac{d\tau}{2\pi i}, \quad j\in\{0,\ldots,k-1\},
\]
where $\A_\rho:=\{\rho<|z|<1/\rho\}$ with positively oriented boundary $\partial\A_\rho$, and we used the Fubini-Tonelli theorem. By taking linear combinations of the previous equations, we deduce from the Cauchy formula that
\[
0 = \int_{\partial\A_\rho}\frac{(ff^\sigma)(\tau)}{(\tau-\xi)^l}\frac{d\tau}{2\pi i} = \frac{(ff^\sigma)^{(l-1)}(\xi)}{(l-1)!}, \quad l\in\{1,\ldots,k\}.
\]
Hence $v$ divides $ff^\sigma$, when viewed as a holomorphic function in $\A_\rho$. Consequently, since $\zeta\in\T$ is a zero of $f$ if and only if it is a zero $f^\sigma$, we get that $f$ vanishes at $\xi$ with multiplicity $\lfloor(k+1)/2\rfloor$, where $\lfloor x \rfloor$ is the integer part of $x$.

\noindent
Next we consider the case where $q$ is a critical point of $\map_n$ having exactly one root on $\T$: $q=vq^*$ with $v(z)=(z-\xi)^k$, $\xi\in\T$, and $q^*\in\mpoly_{n-k}$. Denote by $\mathcal{Q}$ and $\mathcal{V}$ some neighborhoods of $q^*$ and $v$, in $\mpoly_{n-k}^{1/\rho}$ and $\mpoly_k^{1/\rho}$ respectively, taking them so small that each $\chi\in\mathcal{Q}$ is coprime to each $\nu\in\mathcal{V}$; this is possible since $q^*$ and $v$ are coprime. Then, $(\chi,\nu)\mapsto\chi\nu$ is a diffeomorphism from $\mathcal{Q}\times\mathcal{V}$ onto a neighborhood of $q$ in $\mpoly_n^{1/\rho}$. In particular, the fact that $q$ is a critical point of $\map_{n}$ means that $q^*$ is a critical point of $\Theta$ and $v$ a critical point of $\Xi$, where
\[
\begin{array}{ccc}
\begin{array}{rll}
\Theta : \mathcal{Q} &\to&  [0,\infty) \\
\chi &\mapsto& \map_{n}(v \chi)
\end{array}
&\mbox{and}&
\begin{array}{rll}
\Xi : \mathcal{V} &\to&  [0,\infty) \\
\nu &\mapsto& \map_{n}(\nu q^*).
\end{array}
\end{array}
\]
Since $\widetilde{v\chi}=e^{i\theta}v\widetilde{\chi}$, where $e^{i\theta}$ is as in (\ref{eq:lquqT}), it follows from (\ref{eq:uq}) that $u_{v\chi}=e^{-i\theta}u_{\chi}$, and therefore by (\ref{eq:anlook}) that $\Theta = \map_{n-k|\mathcal{Q}}$, implying that $q^*$ is a critical point of $\map_{n-k}$. In another connection, shrinking $\mathcal{V}$ if necessary, we may assume there exists $\varrho>\rho$ such that $\mathcal{V}\subset\mpoly_k^{1/\varrho}$.  Put for simplicity $w:=u_{q^*}=u_{f,q^*}$, which is clearly an element of $\bar H_0$ by (\ref{eq:uq}). Computing with the latter formula yields for any $\nu\in\mathcal{V}$ that
\begin{eqnarray}
u_{w,\nu}(z) &=& \frac{1}{2\pi i}\int_{\T_\varrho} \frac{(w\nu)(\tau)}{\widetilde \nu(\tau)}\frac{d\tau}{z-\tau}  = \frac{1}{2\pi i} \int_{\T_\varrho} \frac{1}{2\pi i}\int_{\T_\rho} \frac{(fq^*)(t)}{\widetilde q^*(t)}\frac{dt}{\tau-t} \frac{\nu(\tau)}{\widetilde \nu(\tau)}\frac{d\tau}{z-\tau}  \nonumber \\
{} &=& \frac{1}{2\pi i}\int_{\T_\rho} \frac{(f\nu q^*)(t)}{\widetilde{(\nu q^*)}(t)}\frac{dt}{z-t} = u_{f,\nu q^*}(z), \quad |z|>\varrho, \nonumber
\end{eqnarray}
where we used the Fubini-Tonelli theorem and the Cauchy integral formula. Thus, we derive from (\ref{eq:anlook}) that
\[
\Xi(\nu) = \map_{f,n}(\nu q^*) = \|u_{f,\nu q^*}\|_2^2 = \|u_{w,\nu}\|_2^2 = \map_{w,k}(\nu), \quad \nu\in\mathcal{V}.
\]
As $v$ is a critical point of $\Xi$, we see that it is also critical for $\map_{w,k}$, so by the case previously considered we conclude that $w=u_{f,q^*}$ vanishes at $\xi$ with multiplicity $\lfloor(k+1)/2\rfloor$. By (\ref{eq:ip}), this is equivalent to the fact that $L_{q^*}/q^*=L_q/q$  interpolates $f$ at the zeros of $d(z) = (z-\xi)^{\lfloor(k+1)/2\rfloor}$ in Hermite's sense.

Finally, the case where $q$ is arbitrarily located on $\partial\overline\mpoly_n$ is handled the same way upon writing $q=q^*v_1\ldots v_\ell$, where $v_j(z)=(z-\xi_j)^{k_j}$ for some $\xi_j\in\T$, and introducing a product neighborhood $\mathcal{Q}\times\mathcal{V}_1\times\ldots\times\mathcal{V}_\ell$ of $q^*v_1\ldots v_\ell$ to proceed with the above analysis on each of the corresponding maps $\Theta$, $\Xi_1,\ldots,\Xi_\ell$. Thus, taking into account Proposition \ref{prop:2} and the fact that $v_j$ and $\check v_j$ have the same zeros in $\overline{\C}$, we obtain:
\begin{prop}
\label{prop:4}
Let $f\in \bar H_0$ and $q=vq^*$, where $v=\prod (z-\xi_j)^{k_j}$, $\xi_j\in\T$, $\deg(v)=k$, and $q^*\in\mpoly_{n-k}$. Assume that $q$ is a critical point of $\map_n=\map_{f,n}$. Then $q^*$ is a critical point of $\map_{n-k}$. Moreover, if we write $q^*=q_1d_1$ where $d_1$ is  the monic g.c.d. of $L_{q^*}$ and $q^*$, then $L_{q^*}/q^* = L_q/q$ interpolates $f$ at the zeros of $\check{q_1}^2\check{d_1}\check{d}/z$ in Hermite's sense on $\overline{\C}\setminus\D$, where $d(z) = \prod(z-\xi_j)^{\lfloor(k_j+1)/2\rfloor}$. 
\end{prop}

Again the converse of Proposition \ref{prop:4} is true, namely the properties of $q^*$ and $v$ asserted there imply that $q=q^* v$ is critical for $\map_n$. This is easy to check by reversing the previous arguments, but we shall not use it.

\section{A Criterion for Local Minima}
\label{sec:criterion}

% A criterion for a local minimum first was obtained in \cite[Thm. 2.12]{BSW96} and further refined in \cite{BStW99} to the present formulation. However, as with the theory of critical points, it was done only for the case of real Hardy spaces. Despite the simplicity of the adjustments, the authors felt compel to make a  complete exposition of the material for the ease of the reader.
Let $f\in\bar H_0$ and $\map_n=\map_{f,n}$ be the extended map obtained in Proposition \ref{prop:3}, based on (\ref{eq:anlook}) and (\ref{eq:uq}). The latter is a smooth real-valued function, defined on an open neighborhood of $\overline\mpoly_n$ identified with a subset of $\C^n\sim\R^{2n}$ by taking as coordinates all coefficients but the leading one. By definition, a critical point of $\map_n$ is a member of $\overline\mpoly_n$ at which the gradient $\nabla\map_n$ vanishes. This notion is of course independent of which coordinates are used, and so is the signature of the second derivative, the so called \emph{Hessian} quadratic form\footnote{This is not true at non-critical points.}. A critical point $q$ is called nondegenerate if the Hessian form is nonsingular at $q$,  and then the number of negative eigenvalues of this form is called the \emph{Morse index} of $q$ denoted by $M(q)$. Observe that nondegenerate critical points are necessarily isolated.

From first principles of differential topology \cite{GuilleminPollack} it is known that $(-1)^{M(q)}$, which is called the \emph{index} of the nondegenerate critical point $q$,  is equal to the so-called Brouwer degree of the vector field $\nabla\map_n/\|\nabla\map_n\|_\textnormal{e}$ on any sufficiently small sphere centered at $q$, where $\|\cdot\|_\textnormal{e}$ is the Euclidean norm in $\R^{2n}$.

One can show that $\partial\overline\mpoly_n$ is a compact manifold\footnote{We skim through technical difficulties here, because this manifold is not smooth; the interested reader should consult the references we give.}, so if $\map_n$ has no critical points on $\partial\overline\mpoly_n$ and only nondegenerate critical points in $\mpoly_n$, then the sum of the indices of the critical points is equal to the Brouwer degree of $\nabla\map_n/\|\nabla\map_n\|_\textnormal{e}$ on $\partial\overline\mpoly_n$. The surprising fact is that the latter is \emph{independent of $f$} (see \cite{thB}, \cite[Sec. 5]{BO88}, and \cite[Thm. 2]{B_CMFT99}) and is actually equal to 1. Altogether, the following analogue of the Poincar\'e-Hopf theorem holds in the present setting.
\begin{thmI}
Let $f\in\bar H_0$ and $\mathscr{C}_{f,n}$ be the set of the critical points of $\map_{f,n}$ in $\overline\mpoly_n$. Assume that all members of $\mathscr{C}_{f,n}$ are nondegenerate, and that $\mathscr{C}_{f,n}\cap\partial\overline\mpoly_n=\emptyset$. Then
\[
\sum_{q\in\mathscr{C}_{f,n}}(-1)^{M(q)} = 1.
\]
\end{thmI}

To us, the value of the index theorem is that if can show every critical point  is a nondegenerate local minimum and none of them lies on $\partial\overline{\mpoly_n}$, then the critical point is unique. To see this, observe that local minima have Morse index 0 and therefore index 1.

To make this criterion effective, we need now to analyze the Morse index of a critical point, starting with the computation of the Hessian quadratic form.

Let $q$ be a critical point of $\map_n$. It is easy to check that the Hessian quadratic form of $\map_n$ at $q$ is given by
\begin{equation}
\label{defHessian}
\mathscr{Q}(v) = 2\re\left(\sum_{j=0}^{n-1}\sum_{k=0}^{n-1}\left(v_jv_k\frac{\partial^2\map_n}{\partial q_k\partial q_j}(q) + v_j\bar v_k\frac{\partial^2\map_n}{\partial\bar q_k\partial q_j}(q)\right)\right),
\end{equation}
where we have set $v(z)=\sum_{j=0}^{n-1}v_jz^j$ for a generic element of $\poly_{n-1}$, the latter being naturally identified with the tangent space to $\mpoly_n$ at $q$, and we continue to consider $q_j,\bar q_j$, $j\in\{0,\ldots,n-1\}$, as coordinates on $\mpoly_n$. Clearly, $q$ is a nondegenerate local minimum if and only if $\mathscr{Q}$ is positive definite, {\it i.e.},
\begin{equation}
\label{eq:pdqf}
\mathcal{Q}(v) >0  \quad \mbox{ for}\quad v\in\poly_{n-1},\ v\neq0.
\end{equation}
Let us assume that $q$ is irreducible, hence $q\in\mpoly_n$ by Proposition \ref{prop:4}. To derive conditions that ensure the validity of (\ref{eq:pdqf}), we commence by reworking the expression for $\mathscr{Q}$.

Any polynomial in $\poly_{2n-1}$ can be written $p_1L_q+p_2q$ for suitable $p_1,p_2\in\mpoly_{n-1}$, due to the coprimeness of $L_q$ and $q$. Therefore
\begin{equation}
\label{eq:ortherror}
\left\langle\frac{p}{q^2},f-\frac{L_q}{q}\right\rangle = 0 \quad \mbox{for any} \quad p\in\poly_{2n-1}
\end{equation}
by (\ref{eq:orth}) and (\ref{eq:partials}). In view of (\ref{eq:ortherror}),  differentiating (\ref{eq:orth}) with respect to $q_j$ and evaluating at $q$ leads us to 
\begin{equation}
\label{eq:moreorth}
\left\langle \frac{\partial}{\partial q_j}\left(\frac{L_q}{q}\right),\frac{p}{q}\right\rangle = 0, \quad j\in\{0,\ldots,n-1\}, \quad p\in\poly_{n-1},
\end{equation}
which means that $\partial(L_q/q)/\partial q_j$ belongs to $V_q^\perp$. Hence, we get from (\ref{eq:vqperp}) that
\begin{equation}
\label{eq:nuj}
\frac{\partial}{\partial q_j}\left(\frac{L_q}{q}\right) = \frac{q\partial L_q/\partial q_j-z^jL_q}{q^2} =: \frac{\widetilde q\nu_j}{q^2}, \quad \nu_j\in\poly_{n-1}, \quad j\in\{0,\ldots,n-1\}.
\end{equation}
As $L_q$ and $q$ are coprime, the polynomials $\nu_j$ are linearly independent by construction,  thus we  establish a one-to-one linear correspondence on $\poly_{n-1}$ by setting
\begin{equation}
\label{correspnu}
v(z) = \sum_{j=0}^{n-1}v_jz^j \quad \leftrightarrow \quad \nu(z) = - \sum_{j=0}^{n-1}v_j\nu_j(z).
\end{equation}
% which is such that 
% \begin{equation}
% \label{eq:vlb}
% \|v\|_2 \leq \textnormal{const.}\|\nu/q\|_2 \quad \mbox{for any} \quad v\in\poly_{n-1}.
% \end{equation}
Moreover, from Proposition \ref{prop:2} where $d=1$ and $q^*=q$, we can write ({\it compare} (\ref{interpw}))
\begin{equation}
\label{eq:errzout}
f - \frac{L_q}{q} = \frac{\widetilde q\check q}{q}w_q^\sigma \quad \mbox{for some} \quad w_q\in H^2.
\end{equation}
Note that $w_q^\sigma\in \bar H_0$ since $f$ does, hence $w_q$ is holomorphic across $\T$.
\noindent
Now, it follows from (\ref{eq:partials}) and (\ref{dconj}) that
\begin{eqnarray}
\frac{\partial^2 \map_n}{\partial\bar q_k\partial q_j}(q) &=& -\left\langle\frac{\partial^2}{\partial\bar q_k\partial q_j}\left(\frac{L_q}{q}\right),f-\frac{L_q}{q}\right\rangle + \left\langle\frac{\partial}{\partial q_j}\left(\frac{L_q}{q}\right),\frac{\partial}{\partial q_k}\left(\frac{L_q}{q}\right)\right\rangle \nonumber \\
\label{eq:secpar1}
&=& \left\langle\frac{\widetilde q\nu_j}{q^2},\frac{\widetilde q\nu_k}{q^2} \right\rangle = \left\langle\frac{\nu_j}{q},\frac{\nu_k}{q} \right\rangle
\end{eqnarray}
by (\ref{eq:ortherror}), (\ref{eq:nuj}), and the fact that $\widetilde q/q$ is unimodular on $\T$. Furthermore
\begin{eqnarray}
\frac{\partial^2 \map_n}{\partial q_k\partial q_j}(q) &=& -\left\langle\frac{\partial^2}{\partial q_k\partial q_j}\left(\frac{L_q}{q}\right),f-\frac{L_q}{q}\right\rangle + \left\langle\frac{\partial}{\partial q_j}\left(\frac{L_q}{q}\right),\frac{\partial L_q/\partial\bar q_j}{q}\right\rangle \nonumber \\
&=& -\left\langle\frac{\partial^2}{\partial q_k\partial q_j}\left(\frac{L_q}{q}\right),\frac{\widetilde q\check q}{q}w_q \right\rangle \nonumber
\end{eqnarray}
by (\ref{eq:moreorth}) and (\ref{eq:errzout}). Now, a simple computation using (\ref{eq:nuj}) yields
\[
\frac{\partial^2}{\partial q_k\partial q_j}\left(\frac{L_q}{q}\right) = \frac{q(\partial^2L_q/\partial q_k\partial q_j)-z^k(\partial L_q/\partial q_j)+z^j(\partial L_q/\partial q_k)}{q^2} - 2z^j\frac{\widetilde q\nu_k}{q^3},
\]
and since the first fraction on the above right-hand side belongs to $\poly_{2n-1}/q$, we deduce from (\ref{eq:moreorth}) and what precedes that
\begin{equation}
\label{eq:secpar2}
\frac{\partial^2 \map_n}{\partial q_k\partial q_j}(q) = 2\left\langle\frac{z^j\widetilde q\nu_k}{q^3},\frac{\widetilde q\check q}{q}w_q^\sigma\right\rangle = 2\left\langle\frac{z^j\nu_k}{q},w_q^\sigma\right\rangle,
\end{equation}
since $\widetilde q/q$ is unimodular while $\overline{\check q} = q$ on $\T$. So, we get from (\ref{eq:secpar1}), (\ref{eq:secpar2}), and (\ref{correspnu}) that
\[
\sum_{j=0}^{n-1}\sum_{k=0}^{n-1}v_jv_k \frac{\partial^2\map_n}{\partial q_k\partial q_j}(q) = -2\left\langle\frac{v\nu}{q},w_q^\sigma\right\rangle = -2\left\langle\frac{\nu}{q},(vw_q)^\sigma\right\rangle
\]
and
\[
\sum_{j=0}^{n-1}\sum_{k=0}^{n-1}v_j\bar v_k \frac{\partial^2\map_n}{\partial\bar q_k\partial q_j}(q) = \left\langle\frac{\nu}{q},\frac{\nu}{q}\right\rangle = \left\|\frac{\nu}{q}\right\|^2_2.
\]
Therefore, in view of (\ref{defHessian}) the quadratic form $\mathscr{Q}/2$ can be rewritten as
\begin{equation}
\label{reecrQ}
\frac12\mathscr{Q}(v) = \left\|\frac{\nu}{q}\right\|^2_2 - 2\re\left\langle\frac{\nu}{q},(vw_q)^\sigma\right\rangle = \left\|\frac{\nu}{q}\right\|^2_2 - 2\re\left\langle\left(\frac{\nu}{q}\right)^\sigma,vw_q\right\rangle.
\end{equation}

To manage the above expression,  we assume that $L_q/q$ does not interpolate $f$ on $\T$, i.e. that $w_q$ has no zeros there, and we let $Q\in\mpoly_l$ have the same zeros as $w_q$ in $\D$, counting multiplicities. Thus we can write $w_q=o_qQ/\widetilde{Q}$, where $o_q$ is holomorphic and zero-free on a neighborhood of $\overline{\D}$, while $|o_q|=|w_q|$ on $\T$ since $Q/\widetilde{Q}$ is unimodular there\footnote{The function $o_q$ is none but the \emph{outer factor} of $w_q$ in $H^2$, see \cite[Thm. 2.8]{Duren}.}. Consider now the Hankel operator $\Gamma$, with  symbol $s_q:=L_q/(o_qq\widetilde q)$, i.e.
\[
\begin{array}{rll}
\Gamma: H^2 &\to& \bar H_0^2 \\
u &\mapsto& \displaystyle \pj_-\left(s_qu\right),
\end{array}
\]
where $\pj_-$ is the orthogonal projection from $L_2$ onto $\bar H_0^2$. Observe that $\Gamma$ is well defined because $s_q$ is bounded on $\T$, and since the latter is meromorphic in $\D$ with poles at the zeros of $q$, counting multiplicities. It is elementary \cite{Peller} that $\Gamma(H^2)=V_q$, that ${\rm Ker}\,\Gamma=(q/\widetilde{q})H^2$, and that $\Gamma: V_{\widetilde q} \to V_q$ is an isomorphism, where $V_{\widetilde q}:=\poly_{n-1}/\widetilde q$ is readily seen to be the orthogonal complement of ${\rm Ker}\,\Gamma$ in $H^2$. Thus, there exists an operator $\Gamma^\#:V_q\to V_{\widetilde q}$, which is inverse to $\Gamma_{|V_{\widetilde q}}$. To evaluate $\Gamma^\#$, observe from (\ref{eq:nuj}) that
\begin{eqnarray}
\Gamma(vo_q) &=& \pj_-\left(\frac{vL_q}{q\widetilde q}\right) = \sum_{j=0}^{n-1}v_j\pj_-\left(\frac{z^jL_q}{q\widetilde q}\right) = \sum_{j=0}^{n-1}v_j\pj_-\left(\frac{q(\partial L_q/\partial q_j)-\widetilde q\nu_j}{q\widetilde q}\right) \nonumber \\
&=& \sum_{j=0}^{n-1}v_j\pj_-\left(\frac{-\widetilde q\nu_j}{q\widetilde q}\right) = \pj_-\left(\frac{\nu}{q}\right) = \frac{\nu}{q}, \nonumber
\end{eqnarray}
where we used that $(\partial L_q/\partial q_j)/\widetilde q\in H^2$ and that $\nu/q\in\bar H_0^2$. Hence we may write
\begin{equation}
\label{relev}
\Gamma^\#\left(\frac{\nu}{q}\right) = vo_q + u, \quad \mbox{with}\quad u\in 
\frac{q}{\widetilde q}\,H^2={\rm Ker}\,\Gamma,
\end{equation}
and since $w_q/o_q\in H^\infty$ it follows that $uw_q/o_q\in {\rm Ker}\,\Gamma$ as well, entailing by (\ref{reecrQ}) and (\ref{relev}) that
\[
\frac12\mathscr{Q}(v) = \left\|\frac{\nu}{q}\right\|^2_2 - 2\re\left\langle\left(\frac{\nu}{q}\right)^\sigma,\frac{w_q}{o_q}\,\Gamma^\#\left(\frac{\nu}{q}\right)\right\rangle
\]
because $(\nu/q)^\sigma\in V_{\widetilde q}=({\rm Ker}\,\Gamma)^\perp$.  Altogether, we see that
\[
\frac12\mathscr{Q}(v)  \geq  \left\|\frac{\nu}{q}\right\|^2_2 - 2\left|\left\langle\left(\frac{\nu}{q}\right)^\sigma,\frac{w_q}{o_q}\,\Gamma^\#\left(\frac{\nu}{q}\right)\right\rangle\right| \geq \left(1-2\|\Gamma^\#\|\right)\left\|\frac{\nu}{q}\right\|_2^2 \nonumber
\]
by the Schwarz inequality and since $|w_q/o_q|=1$ on $\T$ while the $\sigma$ operation preserves the norm. The inequalities above imply that $\mathscr{Q}$ is positive definite as soon as $\|\Gamma^\#\|<1/2$. This last inequality is equivalent to saying that 2 is strictly less than the smallest singular value of $\Gamma_{|_{V_q}}$, which is also  the $n$-th singular value of $\Gamma$ since $V_{\widetilde q}$ has dimension $n$ and is the orthogonal complement of $\mbox{Ker}\,\Gamma$ in $H^2$. By the Adamjan-Arov-Krein theorem \cite{Peller}, the singular value in question is equal to the error in $L^\infty$-best approximation to $s_q$ from $H^\infty_{n-1}$, where $H^\infty_{n-1}$ stands for the set of functions of the form $h/\chi$ where $h\in H^\infty$ and $\chi\in\mpoly_{n-1}$.  Let us indicate this approximation number by $\sigma_{n-1}$:
\[
\sigma_{n-1} := \inf_{g\in H^\infty_{n-1}}\left\|s_q- g\right\|_{L^\infty}.
\]
As $s_q$ is holomorphic on a neighborhood of $\overline{\D}$, it follows from the Adamjan-Arov-Krein theory that the infimum is uniquely attained at some $g_{n-1}\in H^\infty_{n-1}$ which is holomorphic on a neighborhood of $\T$, that $|s_q-g_{n-1}|(\xi)=\sigma_{n-1}$ for all $\xi\in\T$, and that $\wn\left(s_q - g_{n-1}\right) \leq -2n+1$ as soon as $\sigma_{n-1}>0$, where $\wn$ stands for the usual winding number of a non-vanishing continuous function on $\T$. 

We will appeal to a de la Vall\'ee-Poussin principle for this type of approximation, to the effect that
\begin{equation}
\label{dVP}
\sigma_{n-1} \geq \inf_\T\left|s_q- g\right|,
\end{equation}
whenever $g\in H_{n-1}^\infty$ is such that
\[
\wn\left(s_q - g\right) \leq 1-2n.
\]
This principle is easily deduced from the Rouch\'e theorem, for if  (\ref{dVP}) did not hold then the inequality
\[
|(g_{n-1}-g)-(s_q-g)|=|g_{n-1}-s_q|=\sigma_{n-1}<|s_q-g|
\]
would imply that $\wn(g_{n-1}-g)=\wn(s_q-g)\leq 1-2n$, which is impossible unless $g_{n-1}=g$  because $g_{n-1}-g$ is meromorphic with at most $2n-2$ poles in $\D$.

Hence, with our assumptions, that $q$ is an irreducible critical point and that $f-L_q/q$ has no zero on $\T$, we find that $\mathscr{Q}$ will be positive definite if there exists $\Pi_q\in\rat_{n-1}$ such that
\begin{equation}
\label{critcomp}
2|f-L_q/q| < |\Pi_q - L_q/q| \quad \mbox{on} \quad \T \quad \mbox{and} \quad \wn(f-\Pi_q)\leq 1-2n.
\end{equation}
Indeed, in this case, we will get
\begin{equation}
\label{crit2}
2 < \left|\frac{L_q/q-\Pi_q}{\widetilde q\check q w_q^\sigma/q}\right| = \left|\frac{L_q}{o_qq\widetilde q}-\frac{\Pi_q}{\widetilde qo_q}\right|=
\left|s_q-\frac{\Pi_q}{\widetilde qo_q}\right| \quad \mbox{on} \quad \T,
\end{equation}
because $|\check q/q|\equiv1$ and $|w_q^\sigma|\equiv|w_q|\equiv|o_q|$ on $\T$. Moreover, it follows from (\ref{critcomp}) and the triangle inequality that $|f-L_q/q|<|\Pi_q-f|$, and therefore
\[
\left|(L_q/q-\Pi_q)-(f-\Pi_q) \right|<\left|f-\Pi_q \right|
\]
so that $\wn(L_q/q-\Pi_q)=\wn(f-\Pi_q)$ by Rouch\'e's theorem. Consequently
\[
\wn\left(s_q-\frac{\Pi_q}{\widetilde qo_q}\right)=\wn\left(\frac{L_q/q-\Pi_q}{o_q\widetilde q}\right) = \wn\left(L_q/q-\Pi_q\right) = \wn(f-\Pi_q),
\]
where we used  that $o_q\widetilde q$ does not vanish on $\overline{\D}$. So we see that $\Pi_q/\widetilde{q}o_q$ can be used as $g$ in (\ref{dVP}) to bound $\sigma_{n-1}$ from below by a quantity which, by (\ref{crit2}), is strictly bigger than 2. Altogether, rewriting (\ref{critcomp}) in the equivalent form (\ref{eq:p1}) below, we proved:

\begin{thm}[Comparison Criterion]
\label{thm:U}
Let $f\in\bar H_0$ and $q\in\mpoly_n$ be an irreducible critical point of $\map_{f,n}$ such that $f-L_q/q$ does not vanish on $\T$. Then $q$ is a nondegenerate local minimum as soon as there exists $\Pi_q\in\rat_{n-1}$ satisfying
\begin{equation}
\label{eq:p1}
2 < \left|1 - \frac{f-\Pi_q}{f-L_q/q}\right| \quad \mbox{on} \quad \T \quad \mbox{and} \quad \wn(f-\Pi_q)=1-2n.
\end{equation}
\end{thm}

In order to use this criterion, we need an appraisal of the error in interpolation to $f$ by members of $\rat_n$ and $\rat_{n-1}$. In the next section, we gather the necessary estimates for the class of functions introduced in Theorem \ref{thm:main} after the works \cite{uBY3,uY2}.

\section{Error in Rational Interpolation}
\label{sec:interpol}

Let us recall the notion of a diagonal multipoint Pad\'e approximant. Conceptually, a diagonal (multipoint)  Pad\'e approximant to a function $g$ holomorphic in a domain $\Omega\subset\overline{\C}$, is a rational function of type $(n,n)$\footnote{A rational function is said to be of type $(n_1,n_2)$ if it can be written as the ratio of a polynomial of degree at most $n_1$ by a polynomial of degree at most $n_2$.} that interpolates $g$ in a prescribed system of $2n+1$ points of $\Omega$, counting multiplicities. However, such a definition may not work and it is best to adopt a linearized one as follows. Without loss of generality, we normalize one interpolation point to be infinity and assume that $g(\infty)=0$. The remaining $2n$ interpolation points, finite or infinite, form a set $\mathcal{I}_{2n}$ accounting for multiplicities with repetition. Let $Q_{2n}$ be a  polynomial vanishing exactly at the finite points of $\mathcal{I}_{2n}$. Then: 

\begin{df}[Pad\'e Approximant] The {\it n-th diagonal (multipoint) Pad\'e approximant} to $g$ associated with $\mathcal{I}_{2n}$ is the rational function $\Pi_n=p_n/\ell_n$ satisfying:
\begin{itemize}
\item $\deg p_n\leq n$, $\deg \ell_n\leq n$, and $\ell_n\not\equiv0$;
\item $\left(\ell_n(z)g(z)-p_n(z)\right)/Q_{2n}(z)$ is analytic in $\Omega$;
\item $\left(\ell_n(z)g(z)-p_n(z)\right)/Q_{2n}(z)=O\left(1/z^{n+1}\right)$ 
as $z\to\infty$.
\end{itemize}
\end{df}

The conditions for $p_n$ and $\ell_n$ amount to solving a system of $2n+1$ homogeneous linear equations with $2n+2$ unknown coefficients, and clearly no solution can be such that $\ell_n\equiv0$. Moreover it is plain to see that all pairs $(p_n,\ell_n)$ define the same rational function $p_n/\ell_n$, thus a Pad\'e approximant indeed exists uniquely with the above definition. As a result of our normalization, note that the third condition in the above definition entails at least one interpolation condition at infinity, therefore $\Pi_n$ is in fact of type $(n-1,n)$.

For our present purpose, we shall be interested only in the case where $g=\f$ is as in Theorem \ref{thm:main}, and $\mathcal{I}_{2n}$ consists of $n$ points, each of which appears with multiplicity 2. In other words, we let $E_n$ consist of $n$ points, repeated according to their multiplicities, in the analyticity domain of $\f$, say, $\da=\overline{\C}\setminus([a,b] \cup \Lambda)$ where $[a,b]:=\supp(\mu)$ and $\Lambda$ is the set of poles of $r$. We further let $v_n$ be the monic polynomial whose roots are the finite points of $E_n$, and we obtain $\Pi_n$ from the previous definition where $g=\f$ and $Q_{2n}=v_n^2$.

Next, we let $n$ range over $\N$ and we put $\E=\{E_n\}$  for the {\it interpolation scheme}, i.e. the sequence of sets of interpolation points. By definition, the {\it support} of $\E$ is $\supp(\E):=\cap_{n\in\N}\overline{\cup_{k\geq n}E_k}$. We also introduce the \emph{probability counting measure} of $E_n$ to be the measure with mass $1/n$ at each point of $E_n$, repeating according to multiplicities. 

We shall need strong asymptotics on the behaviour of $\Pi_n$ as  $n\to\infty$. To describe them, we need some more notation. Let us denote by
\[
w(z) := \sqrt{(z-a)(z-b)}, \quad w(z)/z \to 1,
\]
the holomorphic branch of the square root outside of $[a,b]$ which is
positive on $(b,+\infty)$, and by
\begin{equation}
\label{eq:cmap}
\cmap(z) := \frac{2}{b-a}\left(z-\frac{b+a}{2}-w(z)\right)
\end{equation}
the conformal map of $D:=\overline\C\setminus[a,b]$ into $\D$ such that $\cmap(\infty)=0$ and $\cmap^\prime(\infty)>0$. Note that $\cmap$ is conjugate-symmetric.

Recall that the logarithmic energy of a positive Borel measure $\sigma$, compactly supported in $\C$, is given by $-\int\int\log|z-t|d\sigma(z)d\sigma(t)$, which is a real number or $+\infty$.

\begin{df}[Admissibility]
An interpolation scheme $\E$ is called admissible if the sums $\sum_{e\in E_n} |\cmap(e)-\cmap(\bar e)|$ are uniformly bounded with $n$, $\supp(\E)\subset\da$, and the probability counting measure of $E_n$ converges weak$^*$ to some Borel measure $\sigma$ with finite logarithmic energy\footnote{Note that $\sigma$ may not be compactly supported. In this case, pick $z_0\in\C\setminus\supp(\E)$ such that $z_0\notin\supp(\sigma)$ and set $M_{z_0}(z):=1/(z-z_0)$. Then, all the sets $M_{z_0}(E_n)$ are contained in a common compact set and their counting measures converge weak$^*$ to $\sigma^\prime$ such that $\sigma^\prime(B):=\sigma(M_{z_0}^{-1}(B))$ for any Borel set $B\subset\C$. What we require is then the finiteness of the logarithmic energy of $\sigma^\prime$.}.
\end{df}

The weak$^*$ convergence in the above definition is understood upon regarding complex measures on $\C$ as the dual space of continuous functions with compact support. 

To an admissible scheme $\E$, we associate a sequence of functions on $\da$ by putting
\begin{equation}
\label{eq:rn}
R_n(z) = R_n(\E;z) := \prod_{e\in E_n}\frac{\cmap(z)-\cmap(e)}{1-\cmap(z)\cmap(e)}, \quad z\in D.
\end{equation}
Each $R_n$ is holomorphic in $D$, has continuous boundary values 
from both sides of $[a,b]$, and vanishes only at points of $E_n$.
Note from the conjugate-symmetry of $\cmap$ that
\[\frac{\cmap(z)-\cmap(e)}{1-\cmap(z)\cmap(e)}=
\frac{\cmap(z)-\overline{\cmap(e)}}{1-\cmap(z)\cmap(e)}
\left(1+\frac{\cmap({\bar e})-\cmap(e)}{\cmap(z)-\cmap(\bar e)}\right).
\] 
Thus, $R_n$ is a Blaschke product with zero set $\overline{\cmap(E_n)}$ composed with $\cmap$, times an infinite product which is boundedly convergent on any curve separating $[a,b]$ from $\supp(\E)$ by the admissibility conditions. In particular, $\{R_n\}$ converges to zero locally uniformly in $D$.

To describe asymptotic behavior of multipoint Pad\'e approximants, we need two more concepts. Let $h$ be a Dini-continuous function on $[a,b]$. Then the {\it geometric mean} of $h$, given by
\[
\gm_h := \exp\left\{\int\log h(t)d\ed(t)\right\},
\]
is independent of the actual choice of the branch of the logarithm \cite[Sec. 3.3]{uBY3}. Moreover, the {\it Szeg\H{o} function} of $h$, defined as
\[
\szf_h(z) := \exp\left\{\frac{w(z)}{2}\int\frac{\log h(t)}{z-t}d\ed(t) - \frac12\int\log h(t)d\ed(t)\right\}, \quad z\in D,
\]
does not depend on the choice of the branch either (as long as the same branch is taken in both integrals) and is the unique non-vanishing holomorphic function in $D$ that has continuous boundary values from each side of  $[a,b]$ and satisfies $h = \gm_h\szf_h^+\szf_h^-$ and $\szf_h(\infty)=1$. The following theorem was proved in \cite[Thm. 4]{uBY3} when $r=0$ and in \cite{uY2} for the general case.

\begin{thm}
\label{thm:sa}
Let $\f$ be as in Theorem \ref{thm:main}, $\E$ an admissible interpolation scheme, and $\{\Pi_n\}$ the sequence of diagonal Pad\'e approximants to $\f$ associated with $\E$. Then
\begin{equation}
\label{eq:sa}
(\f-\Pi_n)w = [2\gm_{\dot\mu} + o(1)](\szf_{\dot\mu}R_n/R)^2
\end{equation}
locally uniformly in $\da$, where $R_n$ is as in (\ref{eq:rn}) and
\[
R(z):=\prod(\cmap(z)-\cmap(e))/(1-\cmap(z)\cmap(e)),
\]
the product defining $R$ being taken over the poles of $r$ according to their multiplicity.
\end{thm}

Let now $\{q_n\}$ be a sequence of irreducible critical points for $\map_{\f,n}$. Put $q_n(z)=\Pi_{1\leq j\leq n}(z-\xi_{j,n})$. It follows from Proposition \ref{prop:2} that $L_{q_n}/q_n$ interpolates $\f$ at every $1/\bar\xi_{j,n}$ with order \nolinebreak 2, hence $L_{q_n}/q_n$ is the $n$-th diagonal Pad\'e approximants associated with $\E_{\{q_n\}}:=\{\{1/\bar\xi_{j,n}\}_{j=1}^n\}$. This interpolation scheme of course depends on $q_n$, which accounts for the nonlinear character of the $L^2$-best rational approximation problem. The next theorem contains in its statement the Green equilibrium distribution of $\supp(\mu)=[a,b]$, for the definition of which we refer the reader to~\cite{SaffTotik}.

\begin{thm}
\label{thm:cp}
Let $\f$ be as in Theorem \ref{thm:main} and $\{q_n\}$ be a sequence of irreducible critical points for $\f$. Then $\E_{\{q_n\}}$ is an admissible interpolation scheme, and moreover
\begin{equation}
\label{eq:anglem}
\textstyle \sum_{j=1}^n|\im(\xi_{j,n})| \leq \const
\end{equation}
where $\const$ is independent of $n$. Also, the probability counting measures of the zeros of $q_n$ converges to the Green equilibrium distribution on $\supp(\mu)$. In addition, it holds that
\begin{equation}
\label{eq:cp}
(\f-L_{q_n}/q_n)w = [2\gm_{\dot\mu} + o(1)](\szf_{\dot\mu}R_n/R)^2
\end{equation}
locally uniformly in $\da$, where $R_n$ is as in (\ref{eq:rn}) and $R$ is as in Theorem \ref{thm:sa}.
\end{thm}
\noindent
A few comments on Theorem \ref{thm:cp} are in oder. First, the weak$^*$ convergence of the counting measures of the $q_n$ was obtained in \cite[Thm. 2.1]{BY09b}. It entails that the probability counting measures of the sets $E_{\{q_n\}}$ converge weak$^*$ to the reflection of the Green equilibrium measure across $\T$, which has finite energy. The admissibility of $\E_{\{q_n\}}$ follows easily from this and from the bound (\ref{eq:anglem}) which was proven in \cite{BKT05}, see \cite[Lem. 8]{uY2}. Then relation (\ref{eq:cp}) is a consequence of (\ref{eq:sa}).

\section{Proof of Theorem \ref{thm:main}}
\label{sec:proof}

To prove Theorem \ref{thm:main}, we follow the line of argument developed in \cite[Thm. 1.3]{BStW99}. The main difference is that in the present case the critical points are no longer a priori irreducible and their poles no longer belong to the convex hull of the support of the measure. As we shall see, 
these difficulties can be resolved with the help of Theorem \ref{thm:cp}.

\begin{proof}[Proof of Theorem \ref{thm:main}]
We claim there exists $N=N(\f)\in\N$ such that all the critical points of $\map_n=\map_{\f,n}$ in $\overline\mpoly_n$ are irreducible for $n>N$. Indeed, assume to the contrary that there exists an infinite subsequence of reducible critical point, say $\{q_{n_j}\}$. It follows from Propositions \ref{prop:2} and \ref{prop:4} that  each $q_{n_j}$ has a factor $q_{n_j}^*$ such that $q_{n_j}^*\in\mpoly_{n_j-k_{n_j}}$ is an irreducible critical point of $\map_{n_j-k_{n_j}}$, and the difference $\f-L_{q^*_{n_j}}/q^*_{n_j}$ vanishes at the zeros of $\widetilde{q_{n_j}^*}^2d_{n_j}$ where $d_{n_j}$ is a non-constant polynomial of degree at least $\lfloor(k_j+1)/2\rfloor\geq1$ having all its zeros in $\{|z|\geq1\}$. Suppose first that $(n_j-k_{n_j})\to\infty$ as $j\to\infty$. Then, the asymptotic behavior of $\f-L_{q^*_{n_j}}/q^*_{n_j}$ is governed by (\ref{eq:cp}), in particular it can only  vanish at the zeros of $\widetilde{q_{n_j}^*}^2$ for all large $n$ which contradicts the assumption that $d_{n_j}$ is non-constant. Second, suppose that $n_j-k_{n_j}$ remains bounded. Up to a subsequence, we may suppose that  $n_j-k_{n_j}=l$ for some integer $l$. As $\overline\mpoly_l$ is compact, we may assume that $q_{n_j}^*$ converges to some $q\in\overline\mpoly_l$. Since $L_v$ is a smooth function of $v$ in some neighborhood of $\overline\mpoly_l$ (see Subsection \ref{subsec:3}), the polynomials $L_{q^*_{n_j}}$ converge to $L_q$ hence $q_{n_j}^*\f-L_{q_{n_j}^*}$ converges to $q\f-L_q$ locally uniformly in $\da$. In particular, if we pick $0<\rho<1$ such that $\D_\rho$ contains the zeros of $q$, we get that $\f-L_{q_{n_j}^*}/q_{n_j}^*$ is a normal family of functions converging to $\f-L_q/q$ in $|z|>\rho$. But since the number of zeros it has in $\overline{\C}\setminus\D$ increases indefinitely (because it vanishes at the zeros of $d_{n_j}$ which are at least $\lfloor(n-l+1)/2\rfloor$ in number), we conclude that $\f=L_q/q$, which is impossible since $\f$ is not rational. This contradiction proves the claim.

As we just showed, each critical point $q$ of $\map_n$ in $\overline\mpoly_n$, is irreducible for all $n$ large enough, in particular it belongs to $\mpoly_n$  and moreover $L_q/q$ does not interpolate $\f$ on $\T$. Assume further that, for all such $n$, there exists a rational function $\Pi_q\in\rat_{n-1}$  such that (\ref{eq:p1}) holds with $f=\f$. Then $q$ is a local minimum by Theorem \ref{thm:U}, and therefore it is the unique critical point of order $n$ in view of the Index Theorem. Thus, to finish the proof, we need only construct some appropriate function $\Pi_q$ for each critical point $q$ of $\map_n$, provided that $n$ is large enough.

Let $\E_{\{q_n\}}$ be the interpolation scheme induced by $\{q_n\}$ and $\E_\nu$ some admissible interpolation scheme, with $\supp(\E_\nu)\subset\{|z|>1\}$. Set $\{\Pi_n\}$ to be the sequence of diagonal Pad\'e approximants to $\f$ associated with $\E_\nu$. Then Theorems \ref{thm:sa} and \ref{thm:cp} imply that when $n\to\infty$
\begin{equation}
\label{eq:p2}
\frac{(\f-\Pi_{n-1})(z)}{(\f-L_{q_n}/q_n)(z)} = [1+o(1)] \left(\frac{R_{n-1}(\E_\nu;z)}{R_n(\E_{\{q_n\}};z)}\right)^2 \quad \mbox{uniformly on} \quad \T.
\end{equation}
Moreover, for all $n$ large enough, $(\f-\Pi_{n-1})$ is holomorphic outside of $\D$ by (\ref{eq:sa}) and it has $2n-1$ zeros there, namely those of $R_{n-1}(\E_\nu;\cdot)$, counting multiplicities, plus one at infinity. Consequently $\wn(\f-\Pi_{n-1})=1-2n$ for all such $n$, and so (\ref{eq:p1}) will follow from  (\ref{eq:p2}) upon constructing $\E_\nu$  such that, for $n$ large enough,
\begin{equation}
\label{eq:p3}
\left|1- \left(\frac{R_{n-1}(\E_\nu;z)}{R_n(\E_{\{q_n\}};z)}\right)^2 \right| > 2 \quad \mbox{on} \quad \T.
\end{equation}
\noindent
For convenience, let us put $I:=[a,b]=\supp(\mu)$ and  $I^{-1}:=\{x:~1/x\in I\}$, together with $\Omega:=\overline\C\setminus(I\cup I^{-1})$. Set $\varrho:\Omega\to\{1<|z|<A\}$ to be the conformal map such that $\varrho(I)=\T$, $\varrho(I^{-1})=\T_A$,  $\lim_{z\to b^+}\varrho(z)=1$; as is well known, the number $A$ is here uniquely determined by the so-called condenser capacity of the pair $(I,I^{-1})$ \cite{SaffTotik}. Note also that, by construction, $\varrho$ is conjugate-symmetric. Define
\[
h_\nu(z) = \frac{1-(\varrho(z)A)^2}{\varrho^2(z)-A^2}, \quad z\in\Omega,
\]
which is a well-defined holomorphic function in $\Omega:=\overline\C\setminus(I\cup I^{-1})$. It is not difficult to show ({\it cf.} the proof of \cite[Thm 1.3]{BStW99} after eq. (6.24)) that $|1-h_\nu|>2$ on $\T$. Thus, to prove our theorem, it is sufficient to find $\E_\nu$ such that
\begin{equation}
\label{eq:p4}
\left(\frac{R_{n-1}(\E_\nu;z)}{R_n(\E_{\{q_n\}};z)}\right)^2 =  [1+o(1)] h_\nu(z) \quad \mbox{uniformly on} \quad \T.
\end{equation}
For this, we shall make use of the fact, also proven in the course of \cite[Thm 1.3]{BStW99}, that $h_\nu$ can be represented as
\[
h_\nu(z) := \exp\left\{\int\log\frac{1-\cmap(z)\cmap(x)}{\cmap(z)-\cmap(x)}
d\nu(x)\right\},
\] 
where $\nu$ is a signed measure of mass 2 supported on $I^{-1}$.

Denote by $\{\xi_{j,n}\}_{j=1}^n$ the zeros of $q_n$ and by $\{x_{j,n}\}_{j=1}^n$ their real parts. Observe from (\ref{eq:cp}) that any neighborhood of the poles of $r$ which is disjoint from $I$, contains exactly $m$ zeros of $q_n$ for all $n$ large enough. We enumerate these as $\xi_{n-m+1,n},\ldots,\xi_{n,n}$. The rest of the zeros of $q_n$ we order in such a manner that
\[
a < x_{1,n} < x_{2,n} < \ldots < x_{d_n,n} < b,
\]
while those $j\in\{d_n+1,\ldots,n-m\}$ for which $x_{j,n}$ either lies outside of $(a,b)$ or else coincides with $x_{k,n}$ for some $k\in\{1,\ldots,d_n\}$, are numbered arbitrarily. Again from (\ref{eq:cp}), any open neighborhood of $I$ contains $\{\xi_{j,n}\}_{j=1}^{n-m}$ for all $n$ large enough, and therefore
\begin{equation}
\label{eq:del1}
\delta_n^{im} := \max_{j\in\{1,\ldots,d_n\}}|\im(\xi_{j,n})| \to 0 \quad \mbox{as} \quad n\to\infty. 
\end{equation}
In addition, as the probability counting measures of the zeros of $q_n$ converge to a measure supported on the whole interval $I$, namely the Green equilibrium distribution, we deduce that $d_n/n\to1$ and that
\begin{equation}
\label{eq:del2}
\delta_n^{re} := \max\left\{(x_{1,n}-a),(b-x_{d_n,n}),\max_{j\in\{2,\ldots,d_n\}}(x_{j,n}-x_{j-1,n})\right\} \to 0 \quad \mbox{as} \quad n\to\infty.
\end{equation}

Define $\check\nu$ to be the image of $\nu$ under the map $t\to1/t$, so that $\check \nu$ is a signed measure on $I$ of mass 2. Let further $\imap(z):=\cmap(1/z)$ be the conformal map of $\overline{\C}\setminus I^{-1}$ onto $\D$, normalized so that $\imap(0)=0$ and $\imap^\prime(0)>0$, Finally, set
\[
K(z,t) := \log\left|\frac{\imap(z)-\imap(t)}{1-\imap(z)\imap(t)}\right|.
\]
To define an appropriate interpolation scheme $\E_\nu$, we consider the coefficients:
\begin{eqnarray}
c_{1,n}     &:=& \check\nu\left(\left[a,\frac{x_{1,n}+x_{2,n}}{2}\right)\right), \nonumber \\
c_{j,n}      &:=& \check\nu\left(\left[\frac{x_{j-1,n}+x_{j,n}}{2},\frac{x_{j,n}+x_{j+1,n}}{2}\right)\right), \quad j\in\{2,\ldots,d_n-1\}, \nonumber \\
c_{d_n,n} &:=& \check\nu\left(\left[\frac{x_{d_n-1,n}+x_{d_n,n}}{2},b\right]\right). \nonumber
\end{eqnarray}
Subsequently, we define two other sets of coefficients
\[
\left\{
\begin{array}{lcccr}
b_{j,n} &:=&  \displaystyle \sum_{k=1}^j c_{k,n} &=& \displaystyle \check\nu\left(\left[a,\frac{x_{j,n}+x_{j+1,n}}{2}\right)\right) \\ 
a_{j,n} &:=& 2 - b_{j,n} &=& \displaystyle \check\nu\left(\left[\frac{x_{j,n}+x_{j+1,n}}{2},b\right]\right)
\end{array}
\right. \quad j\in\{1,\ldots,d_n-1\},
\]
and $b_{0,n}=a_{d_n,n}:=0$. It follows in a straightforward manner from the definitions that
\[
2 - c_{j,n} = b_{j-1,n} + a_{j,n}, \quad j\in\{1,\ldots,d_n\},
\]
and therefore
\begin{equation}
\label{eq:sum1}
2 \sum_{j=1}^{d_n} K(z,\xi_{j,n}) - \sum_{j=1}^{d_n} c_{j,n} K(z,\xi_{j,n}) = \sum_{j=1}^{d_n-1} b_{j,n} K(z,\xi_{j+1,n}) + \sum_{j=1}^{d_n-1} a_{j,n} K(z,\xi_{j,n}).
\end{equation}
Next, we introduce auxiliary points $y_{j,n}$ by setting
\[
y_{j,n} := \frac{a_{j,n}\xi_{j,n}+b_{j,n}\xi_{j+1,n}}{2}, \quad j\in\{1,\ldots,d_n-1\}.
\]
Observe that
\begin{equation}
\label{eq:yxi}
|y_{j,n}-\xi_{j,n}| = \left|\frac{b_{j,n}}{2}(\xi_{j+1,n}-\xi_{j,n})\right| \leq \frac{\|\check\nu\|}{2}|\xi_{j+1,n}-\xi_{j,n}|,
\end{equation}
where $\|\check\nu\|$ is the total variation of $\check\nu$. 

Let $\mathcal{K}$ be compact in $\Omega$ and $\mathcal{U}\subset\D$ be a neighborhood of $I$ whose closure is disjoint from $\mathcal{K}$. By (\ref{eq:del1}), (\ref{eq:del2}), and (\ref{eq:yxi}) we see that both $\{\xi_{j,n}\}_{j=1}^{d_n}\subset\mathcal{U}$ and $\{y_{j,n}\}_{j=1}^{d_n-1}\subset\mathcal{U}$ for all $n$ large enough. Thus, for such $n$ and $z\in\mathcal{K}$, we can write the first-order Taylor expansions:
\begin{equation}
\label{eq:taylor1}
K(z,\xi_{j,n}) - K(z,y_{j,n}) = \frac{\partial}{\partial t}K(z,y_{j,n})(\xi_{j,n}-y_{j,n}) + O\left((\xi_{j,n}-y_{j,n})^2\right),
\end{equation}
\begin{equation}
\label{eq:taylor2}
K(z,\xi_{j+1,n}) - K(z,y_{j,n}) = \frac{\partial}{\partial t}K(z,y_{j,n})(\xi_{j+1,n}-y_{j,n}) + O\left((\xi_{j+1,n}-y_{j,n})^2\right),
\end{equation}
and adding up (\ref{eq:taylor1}) multiplied by $a_{j,n}$ to (\ref{eq:taylor2}) multiplied by $b_{j,n}$ we obtain
\begin{equation}
\label{eq:oneterm}
b_{j,n} K(z,\xi_{j+1,n}) + a_{j,n} K(z,\xi_{j,n}) - 2K(z,y_{j,n}) = O\left((\xi_{j+1,n}-\xi_{j,n})^2\right),
\end{equation}
where we took (\ref{eq:yxi}) into account and, of course, the three symbols big ``Oh'' used above indicate different functions. By the smoothness of $K$ on $\overline{\C}\setminus I^{-1}\times \overline{\C}\setminus I^{-1}$ and the compactness of $\mathcal{K}\times \overline{\mathcal{U}}$, these big ``Oh'' can be made uniform with respect to $z\in\mathcal{K}$, being majorized by 
\[
\zeta \mapsto 2\|\check\nu\|\sup_{(z,t)\in\mathcal{K}\times\overline{\mathcal{U}}} \left|\frac{\partial^2 K}{\partial t^2}(z,t)\right||\zeta|^2.
\]
In another connection, it is an immediate consequence of (\ref{eq:del1}), (\ref{eq:del2}), and (\ref{eq:anglem}) that
\begin{eqnarray}
\sum_{j=1}^{d_n-1}|\xi_{j+1,n}-\xi_{j,n}|^2 &\leq& \sum_{j=1}^{d_n-1}|x_{j+1,n}-x_{j,n}|^2 + 2\sum_{j=1}^{d_n}|\im(\xi_{j,n})|^2 \nonumber \\
\label{eq:to0}
{} &\leq& (b-a)\delta_n^{re} + \const\delta_n^{im} = o(1).
\end{eqnarray}
Therefore, we derive from (\ref{eq:to0}) upon adding equations (\ref{eq:oneterm}) for $j\in\{1,\ldots,d_n-1\}$ that
\begin{equation}
\label{eq:sum2}
\left|\sum_{j=1}^{d_n-1} b_{j,n} K(z,\xi_{j+1,n}) + \sum_{j=1}^{d_n-1} a_{j,n} K(z,\xi_{j,n}) - 2\sum_{j=1}^{d_n-1} K(z,y_{j,n})\right| = o(1),
\end{equation}
where $o(1)$ is uniform with respect to $z\in\mathcal{K}$. 
In view of (\ref{eq:sum1}), equation (\ref{eq:sum2}) can be rewritten as
\begin{equation}
\label{eq:sum3}
\left|2 \sum_{j=1}^{d_n} K(z,\xi_{j,n}) - 2\sum_{j=1}^{d_n-1} K(z,y_{j,n}) - \sum_{j=1}^{d_n} c_{j,n} K(z,\xi_{j,n})\right| = o(1).
\end{equation}
Now, it follows from (\ref{eq:del2}) and the definitions of $c_{j,n}$ and $h_\nu$ that
\begin{equation}
\label{eq:rs1}
\sum_{j=1}^{d_n} c_{j,n} K(z,x_{j,n}) \to \int K(z,t)d\check\nu(t) = -\log|h_\nu(1/z)| \quad \mbox{as} \quad n\to\infty,
\end{equation}
uniformly with respect to $z\in\mathcal{K}$. Moreover, we deduce from (\ref{eq:del1}) and (\ref{eq:anglem}) that
\begin{eqnarray}
\sum_{j=1}^{d_n} | c_{j,n} ( K(z,\xi_{j,n}) - K(z,x_{j,n}))| &\leq& C \sum_{j=1}^{d_n}|c_{j,n}| |\im(\xi_{j,n})| \nonumber \\
\label{eq:rs2}
&\leq&  C \|\check\nu\|\delta_n^{im} \to 0,
\end{eqnarray}
as $n\to\infty$, where $C=\sup_{(z,t)\in\mathcal{K}\times\overline{\mathcal{U}}}|\partial K/\partial t(z,t)|$. Hence, combining (\ref{eq:rs1}) and (\ref{eq:rs2}) with (\ref{eq:sum3}), we get
\begin{equation}
\label{eq:conv}
2\sum_{j=1}^{d_n-1} K(z,y_{j,n}) - 2\sum_{j=1}^{d_n} K(z,\xi_{j,n}) \to \log|h_\nu(1/z)| \quad \mbox{as} \quad n\to\infty,
\end{equation}
uniformly on $\mathcal{K}$. Define
\[
g_n(z) := \left(\prod_{j=1}^{d_n-1}\frac{\imap(z)-\imap(y_{j,n})}{1-\imap(z)\imap(y_{j,n})} / \prod_{j=1}^{d_n}\frac{\imap(z)-\imap(\xi_{j,n})}{1-\imap(z)\imap(\xi_{j,n})}\right)^2,
\]
which is holomorphic in $\Omega$.  By (\ref{eq:conv}), it holds that $\log|g_n(z)|\to\log|h_\nu(1/z)|$ as $n\to\infty$ uniformly on $\mathcal{K}$, and since the latter was arbitrary in $\Omega$ this convergence is in fact locally uniform there. Thus, $\{g_n\}$ is a normal family in $\Omega$, and any limit point of this family is a unimodular multiple of $h_\nu(1/\cdot)$. However, $\lim_{z\to b^+}h_\nu(z)=1$ while it follows immediately from the properties of $\imap$ that each $g_n$ has a well-defined limit at $1/b$ which is also 1. So, $\{g_n\}$ is, in fact, a locally uniformly convergent sequence in $\Omega$ and its limit is $h_\nu(1/\cdot)$.

Finally, set $\E_{\nu}:=\{E_{\nu,n}\}$, where $E_{\nu,n}=\{\zeta_{j,n}\}$, $\zeta_{j,n}:=1/y_{j,n+1}$ when $j\in\{1,\ldots,d_{n+1}-1\}$, and $\zeta_{j,n}:=1/\bar\xi_{j+1,n+1}$ when $j\in\{d_{n+1},\ldots,n\}$. Then
\[
\left(R_{n-1}(\E_\nu;z)/R_n(\E_{\{q_n\}};z)\right)^2=g_n(1/z)
\]
and (\ref{eq:p4}) follows from the limit just proved that $\{g_n\}\to h_\nu(1/\cdot)$. Thus, it only remains to prove that $\E_\nu$ is admissible. To show the first admissibility condition, put 
\[
X_n := \sum_{j=1}^{n-1}|\cmap(\zeta_{j,n-1})-\cmap(\bar\zeta_{j,n-1})|.
\]
Then, since 
\[
|\im(y_{j,n})|\leq \frac{\|\check\nu\|}{2}\bigl(|\im (\xi_{j,n})| + |\im (\xi_{j+1,n})|\bigr),\qquad  1\leq j\leq d_n-1,
\]
by the very definition of $y_{j,n}$, we get
\begin{eqnarray}
X_n &=& \sum_{j=1}^{d_n-1}|\imap(y_{j,n})-\imap(\bar y_{j,n})|  + \sum_{j=d_n}^{n-1}|\imap(\xi_{j+1,n})-\imap(\bar \xi_{j+1,n})| \nonumber \\
&\leq& 2\sup_{\overline{\mathcal{U}}}
|\imap^\prime|\left(\sum_{j=1}^{d_n-1}|\im(y_{j,n})|+\sum_{j=d_n}^{n-1}|\im(\xi_{j+1,n})|\right)\nonumber \\
&<& 2\sup_{\overline{\mathcal{U}}}
|\imap^\prime| \left(2\|\check\nu\| \sum_{j=1}^{d_n}|\im(\xi_{j,n})|+\sum_{j=d_n+1}^n|\im(\xi_{j,n})|\right) \nonumber
\end{eqnarray}
which is uniformly bounded by (\ref{eq:anglem}). Further, since each $E_{\nu,n}$ is contained in $\mathcal{U}^{-1}$, we have that $\supp(\E_\nu)\subset\{|z|>1\}$. So, it only remains to show that the probability counting measures of $E_{\nu,n}$ converges weak$^*$ to some Borel measure with finite logarithmic energy. Now, since $d_n/n\to1$ as $n\to\infty$, and by the remark made in footnote 3, it is enough to prove that this property holds for the probability counting measures of the points $\{y_{j,n}\}$. But from (\ref{eq:yxi}) and (\ref{eq:to0}), the latter have the same asymptotic  distribution as the points $\{\xi_{j,n}\}$,  namely the Green equilibrium distribution on $I$ by Theorem \ref{thm:cp}. This finishes the proof of Theorem~\ref{thm:main}. 
\end{proof}

\bibliographystyle{plain}
\small
\bibliography{unique}

\end{document}